\documentclass{article}
\title{Heat kernel estimates on spaces with varying dimension}

\usepackage{amsmath,amssymb}
\usepackage{amsthm}
\usepackage{enumitem}
\usepackage{mathrsfs}
\usepackage{tikz}
\usepackage{tikz-3dplot}

\usepackage{hyperref}
\hypersetup{colorlinks=true,urlcolor=blue,linkcolor=blue,citecolor=blue}

\usetikzlibrary{intersections, calc}
\usetikzlibrary{shapes.geometric}

\setlist[enumerate]{label={\upshape (\roman*)}}

\makeatletter
	\@addtoreset{equation}{section}

\newtheorem{Def}{Definition}[section]
\newtheorem{Thm}[Def]{Theorem}
\newtheorem{Lem}[Def]{Lemma}
\newtheorem{Prop}[Def]{Proposition}
\newtheorem{Rem}[Def]{Remark}
\newtheorem{Cor}[Def]{Corollary}

\newcommand{\R}{\mathbb{R}}
\newcommand{\e}{e^{-\rho(x,y)^2/t}}
\newcommand{\ex}{e^{-|x|_{\rho}^2/t}}
\newcommand{\ey}{e^{-|y|_{\rho}^2/t}}
\newcommand{\exy}{e^{-(|x|_{\rho}+|y|_{\rho})^2/t}}
\newcommand{\ee}{e^{-|x-y|^2/t}}

\author{Takumu Ooi\thanks{Research Institute for Mathematical Sciences, Kyoto University, Kyoto, 606-8502, JAPAN.\ E-mail:ooitaku@kurims.kyoto-u.ac.jp}}

\date{}
\begin{document}
\maketitle

\begin{abstract}
 We obtain sharp two-sided heat kernel estimates on spaces with varying dimension, in which two spaces of general dimension are connected at one point. On these spaces, if the dimensions of the two constituent parts are different, the volume doubling property fails with respect to the measure induced by the associated Lebesgue measures. Thus the parabolic Harnack inequalities fail and the heat kernels do not enjoy Aronson type estimates. Our estimates show that the on-diagonal estimates are independent of the dimensions of the two parts of the space for small time, whereas they depend on their transience or recurrence for large time. These are multidimensional version of a space considered by Z.-Q. Chen and S. Lou (Ann. Probab. 2019), in which a 1-dimensional space and a 2-dimensional space are connected at one point.
\end{abstract} 

\section{Introduction}\label{intro}
The heat kernel, the fundamental solution of the heat equation, has been studied in many areas, both for mathematical interest and for its importance in applications. The heat kernel is the transition density of Brownian motion, and it is difficult to determine its explicit form except in some special cases, such as on Euclidean spaces. Thus, heat kernel estimates have been studied on various spaces, see for example, \cite{CKS,G3,GSe,GSc,GIS,S1}.
In a remarkable series of result, Grigor'yan \cite{G3}, Saloff-Coste \cite{S1} and Sturm \cite{St1,St2} proved that the following are equivalent on a metric measure space : $($i$)$ the volume doubling property and scaled Poincar\'{e} inequalities, $($ii$)$ the parabolic Harnack inequalities, $($iii$)$ Aronson type estimates of the heat kernel. These results were extended to the setting of graphs in \cite{D}.

In studies of heat kernel estimates, the volume doubling property is a natural assumption. However, there are many spaces that do not satisfy this property. One such example is a space with varying dimension given as following: for fixed $\varepsilon >0$, 
$$\R_{\varepsilon}^2\cup \R_+\cup \{a^*\}:=\{(x,0)\ |\ x\in\R^2, |x| >\varepsilon \} \cup \{(0,0,x)\ |\ x>0 \} \cup \{a^*\}.$$ Here, we identify $\{x\in \R^2\ |\ |x| \leq \varepsilon \}$ and $0\in \R$ with a point $a^*$. Z.-Q. Chen and S. Lou  \cite{CL} constructed a stochastic process on $\R_{\varepsilon}^2\cup \R_+\cup \{a^*\}$ that they called Brownian motion with varying dimension (BMVD). Note that $ \R_{\varepsilon}^2\cup \R_+\cup \{a^*\}$ was considered instead of $\R^2\cup \R_+$ because 2-dimensional Brownian motion never hits $0$. For BMVD, the following heat kernel estimates were given. To state the result, let $\rho$ be the shortest path metric derived from the Euclidean metric on the two parts of the space (the precise definition is denoted below) and $|x|_{\rho}$ be the distance between $x$ and $a^*$ with respect to $\rho$. 

\begin{Thm}{\rm ({\cite[Theorem 1.3,\ 1.4]{CL}})}\label{small12}
$[{\rm I}]$
Let $T>0$ be fixed. The transition density $p(t,x,y)$ of BMVD satisfies the following estimates when $t\in (0,T].$\\
$($i$)$ For $x\in \R_+$ and $y\in \R_{\varepsilon}^2\cup \R_+\cup \{a^*\}$, $$p(t,x,y) \asymp \frac{1}{\sqrt{t}}\e.$$
$($ii$)$ For $x,y\in \R_{\varepsilon}^2\cup \{a^*\}$ with $|x|_{\rho}\vee |y|_{\rho}<1$,
\begin{eqnarray*}p(t,x,y)\asymp \frac{1}{\sqrt{t}}\e+\frac{1}{t}\left(1\wedge \frac{|x|_{\rho}}{\sqrt{t}} \right)\left(1\wedge \frac{|y|_{\rho}}{\sqrt{t}} \right)\ee, \end{eqnarray*}
and for $x,y\in \R_{\varepsilon}^2\cup \{a^*\}$ with $|x|_{\rho}\vee|y|_{\rho}\geq1$,
$$p(t,x,y) \asymp \frac{1}{t}\e.$$
$[{\rm II}]$ The transition density $p(t,x,y)$ of BMVD satisfies the following estimates for $t\geq 8.$\\
$($i$)$ For $x,y\in \R_{\varepsilon}^2\cup \{a^*\}$, $$p(t,x,y) \asymp \frac{1}{t}\e.$$
$($ii$)$ For $x\in \R_+$ and $y\in \R_{\varepsilon}^2\cup \{a^*\}$, when $|y|_{\rho}\leq 1$,
$$p(t,x,y)\asymp \frac{1}{t}\left(1+\frac{|x|\log{t}}{\sqrt{t}} \right)\e,$$
and when $|y|_{\rho}>1$, $$p(t,x,y)\asymp \frac{1}{t}\left(1+\frac{|x|}{\sqrt{t}}\log{\left(1+\frac{\sqrt{t}}{|y|_{\rho}}\right)} \right)\e. $$
$($iii$)$ For $x,y\in \R_+$
$$p(t,x,y)\! \asymp \! \frac{\ee}{\sqrt{t}}\! \left(1\wedge \frac{|x|}{\sqrt{t}} \right)\! \! \left(1\wedge \frac{|y|}{\sqrt{t}} \right)\!+\frac{\exy}{t}\! \left( \!1\!+\frac{(|x|\!+|y|)\log{t}}{\sqrt{t}}\right).$$
\end{Thm}

Here and throughout this paper, we use the notation $a\wedge b:=\min{\{a,b\}}$, $a\vee b:=\max{\{a,b\}}$, and for $(t,x,y)\in A\subset [0,\infty)\times (\R_{\varepsilon}^d\cup \R_{\varepsilon'}^{d'}\cup\{a^*\})\times (\R_{\varepsilon}^d\cup \R_{\varepsilon'}^{d'}\cup\{a^*\})$ and non-negative functions $f(t,x,y), g(t,x,y), h(t,x,y)$,
$$fe^{-h} \lesssim ge^{-h}$$
 ($\gtrsim,$ respectively) means that there exist $C>0, c_1>0, c_2>0$, independent of $(t,x,y)\in A$, such that $fe^{-c_1h} \leq  Cge^{-c_2h}$ for $(t,x,y)\in A$ ($\geq,$ respectively). Moreover, 
 $$fe^{-h}\asymp ge^{-h}$$ means that $fe^{-h}\lesssim ge^{-h}$ and $fe^{-h}\gtrsim ge^{-h}$. In computations, constants $C,c$ may change from line to line.

Concerning other work for heat kernel estimates on spaces with varying dimension, S. Lou deduced such for Brownian motion with drift on $\R_{\varepsilon}^2\cup \R_+\cup \{a^*\}$ in \cite{L1} and obtained an explicit expression for the heat kernel of distorted Brownian motion on $\R^3\cup \R_+$ in \cite{L2}. 

In this paper, we estimate the heat kernel for Brownian motion on spaces with general varying dimension.
To introduce the setting more precisely, let $d\geq d' \geq 1$ and $\varepsilon,\ \varepsilon'>0.$ We define $$\R_{\varepsilon}^d:=\{x\in \R^d\ ;\ |x|>\varepsilon \}, \ \R_{\varepsilon'}^d:=\{x\in \R^{d'}\ ;\ |x|>\varepsilon' \},$$ where $|\cdot |$ is the Euclidean norm. For simplicity, set $\R_{\varepsilon}^1:=\R_+:=(0,\infty)$ for all $\varepsilon$. For $\R^d$ and $\R^{d'}$, we identify $\{ x\in \R^d; |x|\leq \varepsilon \}$ and $\{ x\in \R^{d'};|x|\leq \varepsilon' \}$ with a point $a^*$. We will establish heat kernel estimates for Brownian motion on $\R_{\varepsilon}^d \cup \R_{\varepsilon'}^{d'}\cup \{a^*\}$, where $\R_{\varepsilon}^d \cup \R_{\varepsilon'}^{d'}$ means $\{(x,\overbrace{0,\cdots ,0}^{d'})\ |\  x\in \R_{\varepsilon}^d \}\cup \{(\overbrace{0,\cdots,0}^{d},y)\ |\ y\in \R_{\varepsilon'}^{d'}\}$.

We define a neighborhood of $a^*$ as $\{a^*\} \cup (U_1\cap \R_{\varepsilon}^d)\cup (U_2\cap \R_{\varepsilon'}^{d'})$ for some neighborhoods $U_1$ of $\{ x\in \R^d\ ;\ |x|\leq \varepsilon \}$ and $U_2$ of $\{ x\in \R^{d'}\ ;\ |x|\leq \varepsilon' \}$. Moreover, we consider the topology on $\R_{\varepsilon}^d \cup \R_{\varepsilon'}^{d'}\cup \{a^*\}$ induced by the neighborhoods. We denote the Borel $\sigma$-field by $\mathcal{B}:=\mathcal{B}(\R_{\varepsilon}^d \cup \R_{\varepsilon'}^{d'}\cup \{a^*\})$.

For a constant $p>0$, we define $m_p(A):=m^{(d)}(A\cap \R^d)+p\ m^{(d')}(A\cap \R^{d'})$ for $A\in \mathcal{B}$. Here, $m^{(d)}$ is the Lebesgue measure on $\R^d$. In particular, $m_p(\{a^*\})=0$.

We extend the definition of Brownian motion with varying dimension as follows. In Theorem \ref{BMVDexuni}, we will describe the existence and the uniqueness of a process satisfying the following definition.
\begin{Def}\label{BMVD_def}
	Let $d\geq d'\geq 1$, $\varepsilon, \varepsilon'>0$ and $p>0$. Brownian motion with varying dimension $($BMVD in abbreviation$)$ with parameters $(\varepsilon, \varepsilon', p)$ on $\R_{\varepsilon}^d \cup \R_{\varepsilon'}^{d'}\cup \{a^*\}$ is an $m_p$-symmetric diffusion $X=(\{X_t\}, \{\mathbb{P}_x\})$ on $\R_{\varepsilon}^d \cup \R_{\varepsilon'}^{d'}\cup \{a^*\}$ such that$:$
	\begin{enumerate}
	\item its part process on $\R_{\varepsilon}^d$ or $\R_{\varepsilon'}^{d'}$ has the same law as Brownian motion killed upon leaving $\R_{\varepsilon}^d$ or $\R_{\varepsilon'}^{d'}$, respectively,
	\item it admits no killings on $a^*$.
	\end{enumerate}\end{Def}
Throughout the paper, $X=(\{X_t\},\{\mathbb{P}_x \})$ denotes BMVD, $\mathbb{E}_x$ denotes the expectation corresponding to $\mathbb{P}_x$ and $P_tf(x):=\mathbb{E}_x\left(f(X_t)\right)$ for a bounded Borel measurable function $f$. Let $p(t,x,y)$ be the heat kernel with respect to $m_p$ whose existence will be proved in Proposition \ref{Nash}. Let $C_c^\infty$ be the set of all smooth functions with compact support and $\sigma_{K}:=\inf{\{t>0\ |\ X_t\in K \}}$ be the hitting time of $K\in \mathcal{B}$.

Next, we introduce a distance $\rho$ on $\R_{\varepsilon}^d \cup \R_{\varepsilon'}^{d'}\cup \{a^*\}$, as follows,
\begin{equation*} |x|_{\rho}:=\left \{ \begin{split} &|x|-\varepsilon&\ {\rm for}&\ x\in \R_{\varepsilon}^d, \\ &|x|-\varepsilon'&\ {\rm for}&\ x\in \R_{\varepsilon'}^{d'},\\
&\ \ \ \ 0\ \ &\ {\rm for}&\ x=a^*,\end{split} \right.\end{equation*}
$$\rho (x,y):=\left(|x|_{\rho}+|y|_{\rho}\right) \wedge |x-y|\\ \ \ {\rm for}\  x,y\in \R_{\varepsilon}^d\cup \R_{\varepsilon'}^{d'}\cup \{a^*\}.$$
Here, for $x\in \R_{\varepsilon}^d\cup\{a^*\}, y\in \R_{\varepsilon'}^{d'}\cup\{a^*\}$ or $x\in \R_{\varepsilon'}^{d'}\cup\{a^*\}, y\in \R_{\varepsilon}^{d}\cup\{a^*\}$, we define $|x-y|:=\infty$.\\

The following theorems are the main results in this paper.
\begin{Thm}[{\bf Small time estimates}]
 	\label{small}
 	Let $d\geq d' \geq 1$ and $T\geq 1$ be fixed. The heat kernel $p(t,x,y)$ satisfies the following estimates when $t\in (0,T].$
 	\begin{enumerate}
 	\item For $x,y \in \R_{\varepsilon'}^{d'}$ with $|x|_{\rho}\vee |y|_{\rho} \leq 1$, $$p(t,x,y)\asymp \frac{1}{\sqrt{t}}\e+\frac{1}{t^{d'/2}}\left( 1\wedge \frac{|x|_{\rho}}{\sqrt{t}}\right)\left( 1\wedge \frac{|y|_{\rho}}{\sqrt{t}}\right)\ee .$$
 	For $x,y \in \R_{\varepsilon'}^{d'}$ with $|x|_{\rho}\vee |y|_{\rho} > 1$, $$p(t,x,y)\asymp \frac{1}{t^{d'/2}}\e$$
 	\item For $x,y \in \R_{\varepsilon}^{d}$ with $|x|_{\rho}\vee |y|_{\rho} \leq 1$, $$p(t,x,y)\asymp \frac{1}{\sqrt{t}}\e+\frac{1}{t^{d/2}}\left( 1\wedge \frac{|x|_{\rho}}{\sqrt{t}}\right)\left( 1\wedge \frac{|y|_{\rho}}{\sqrt{t}}\right)\ee .$$
 	For $x,y \in \R_{\varepsilon}^{d}$ with $|x|_{\rho}\vee |y|_{\rho} > 1$,\ $$p(t,x,y)\asymp \frac{1}{t^{d/2}}\e$$
 	\item For $x\in \R_{\varepsilon}^{d}\cup \{a^*\}, y\in \R_{\varepsilon'}^{d'}\cup \{a^*\}$, $$p(t,x,y)\asymp \frac{1}{\sqrt{t}}\e .$$
 	\end{enumerate}
\end{Thm}
Note that when $d'=1$ and $d=2$, the estimates are the same as those in Theorem \ref{small12} [I]. Intuitively, if BMVD hits $a^*$, or both $x$ and $y$ are close to $a^*$, a 1-dimensional effect appears in the heat kernel. If either $x$ or $y$ is far from $a^*$, the dimension on which BMVD lives affects the heat kernel. We will prove Theorem \ref{small} in Section \ref{smallsec}.

Concerning large time estimates, we give four theorems depending on the dimensions of the two parts of the space.
 \begin{Thm}[{\bf Large time estimates I}]
 	\label{d'=1}
 	Let $d\geq3, d'=1$ and $T>0$ be large. The heat kernel $p(t,x,y)$ satisfies the following estimates when $T\leq t.$
 	\begin{enumerate}
 	\item For $x,y\in \R_+$ with $|x|_{\rho}\wedge |y|_{\rho} > 1$, $$p(t,x,y)\asymp \frac{|x||y|}{\sqrt{t}(|x|+\sqrt{t})(|y|+\sqrt{t})}\e .$$
 	\item For $x,y\in \R_{\varepsilon}^d$ with $|x|_{\rho}\wedge |y|_{\rho} > 1$, $$p(t,x,y)\asymp \frac{1}{t^{3/2}|x|^{d-2}|y|^{d-2}}\exy+\frac{1}{t^{d/2}}\e .$$
 	\item For $x\in \R_+\cup \{a^*\}, y\in \R_{\varepsilon}^{d}\cup \{a^*\}$ or $x,y\in \R_+$ with $|y|_{\rho} \leq 1$ or $x,y\in \R_{\varepsilon}^d$ with $|x|_{\rho} \leq 1$, $$p(t,x,y)\asymp \left(\frac{1}{t^{d/2}}+\frac{|x|}{t^{3/2}|y|^{d-2}} \right)\e .$$
 	\end{enumerate}
 \end{Thm}
Since 1-dimensional Brownian motion is recurrent and $d$-dimensional Brownian motion is transient for $d\geq 3$, if BMVD starting from a point in $\R_+$ enters $\R_{\varepsilon}^d$ and stays there for a long time, it is likely to escape to infinity. Thus, intuitively $\R_+$ affects the heat kernel more than $\R^d$. We will prove Theorem \ref{d'=1} in Section \ref{lsec1} using the projection.

\begin{Thm}[{\bf Large time estimates II}]
 	\label{d'=2,2}
 	Let $d=d'=2$ and $T$ be large. The heat kernel $p(t,x,y)$ satisfies the following estimates when $T\leq t.$
 	\begin{enumerate}
 	\item For $x,y\in \R_{\varepsilon}^2$ or $x,y\in \R_{\varepsilon'}^2$, $$p(t,x,y)\asymp \frac{1}{t}\e .$$
 	\item  For $x\in \R_{\varepsilon}^{2}\cup \{a^*\}$ and $y\in \R_{\varepsilon'}^{2}\cup \{a^*\}$,\end{enumerate}
 	$$p(t,x,y)\asymp \frac{1}{t}\left( U_t(x)U_t(y)+\frac{U_t(x)\log{|y|}}{\log{(1+t|y|)}}+\frac{U_t(y)\log{|x|}}{\log{(1+t|x|)}} \right) \e .$$
 	Here, $U_t(x):=\frac{1}{\log{(t+|x|)}}+\left(1-\frac{\log{|x|}}{\log{\sqrt{t}}} \right)_+ .$
 \end{Thm}

\begin{Thm}[{\bf Large time estimates III}]
 	\label{d'=2,3}
 	Let $d\geq 3, d'=2$ and $T$ be large. The heat kernel $p(t,x,y)$ satisfies the following estimates when $T\leq t.$
 	\begin{enumerate}
 	\item For $x,y \in \R_{\varepsilon}^{d}$, $$p(t,x,y)\asymp \frac{1}{t(\log{t})^2|x|^{d-2}|y|^{d-2}}\exy+\frac{1}{t^{d/2}}\e .$$
 	\item  For $x,y \in \R_{\varepsilon'}^{2}$, $$p(t,x,y)\asymp \frac{\log{(1+|x|)}\log{(1+|y|)}}{t\log{(1+t|y|)}\log{(1+t|x|)}} \e .$$
 	\item For $x\in \R_{\varepsilon}^{d}\cup \{a^*\}, y\in \R_{\varepsilon'}^{2}\cup \{a^*\}$, $$p(t,x,y)\asymp \left(\frac{1}{t(\log{t})^2|x|^{d-2}}+\frac{H_t(y)}{t^{d/2}} \right)\e .$$
 	\end{enumerate}
 	Here, $H_t(y):=\frac{1}{(\log{(1+|y|)})^2}+\left(\frac{1}{2\log{(1+|y|)}}-\frac{1}{\log{t}} \right)_+ .$
\end{Thm}
For $d=d'=2$, BMVD is recurrent and a 2-dimensional effect appears in the large time estimates.
For $d\geq 3,\  d'=2$, we have a mixed case of recurrent and transient parts of the space. In this case, $\R^2$ affects the heat kernel more than $\R^d$ for a similar reason as in the case of $d\geq 3,\ d'=1$. We will prove Theorem \ref{d'=2,2} and \ref{d'=2,3} in Section \ref{lsec2} using Doob's $h$-transform and the relative Faber-Krahn inequality.

\begin{Thm}[{\bf Large time estimates IV}]
	\label{d'=3}
	Let $d\geq d'\geq3$ and $T$ be large. The heat kernel $p(t,x,y)$ satisfies the following estimates when $T\leq t.$
	\begin{enumerate}
 	\item For $x,y \in \R_{\varepsilon'}^{d'}$, $$p(t,x,y)\asymp \frac{1}{t^{d'/2}}\e .$$
 	\item For $x,y \in \R_{\varepsilon}^{d}$ with $|x|_{\rho}\vee |y|_{\rho} \leq 1$, $$p(t,x,y)\asymp \frac{1}{t^{d'/2}}\e .$$
 	For $x,y \in \R_{\varepsilon}^{d}$ with $|x|_{\rho}\vee |y|_{\rho} > 1$, $$p(t,x,y)\asymp \frac{1}{t^{d'/2}|x|^{d-2}|y|^{d-2}}\exy+\frac{1}{t^{d/2}}\e .$$
 	\item For $x\in \R_{\varepsilon}^{d}\cup \{a^*\}, x\in \R_{\varepsilon'}^{d'}\cup \{a^*\}$, $$p(t,x,y)\asymp \left( \frac{1}{t^{d'/2}|x|^{d-2}}+\frac{1}{t^{d/2}|y|^{d-2}}\right) \e .$$
	\end{enumerate}
\end{Thm}

For $d\geq d'\geq 3$, both Brownian motion on $\R^d$ and $\R^{d'}$ are transient. Intuitively, $d\geq d'$ yields that  $d$-dimensional Brownian motion escape to infinity faster than $d'$-dimensional Brownian motion. Thus, $\R^d$ affects the large time heat kernel more than $\R^{d'}$ if BMVD starts near $a^*$. We will prove Theorem \ref{d'=3} in Section \ref{lsec3} by estimating $p(t,a^*,a^*)$ and using $\mathbb{P}_x(\sigma_{a^*}\in ds)$.

In related works, A. Grigor'yan, L. Saloff-Coste and S. Ishiwata obtained heat kernel estimates for Brownian motion on the connected sum of manifolds \cite{GSc,GIS}. To explain their results, we present the following definition. 

\begin{Def}
Let $M_1$ and $M_2$ be n-dimensional manifolds. A connected sum $M:=M_1\# M_2$ is a manifold constructed by removing a ball inside each manifold and gluing together these boundary spheres. A non-empty compact set $K \subset M$ is a central part of $M$ if the exterior $M\setminus K$ is a disjoint union of open sets $E_1$ and $E_2$ such that each $E_i$ is homeomorphic to $M_i\setminus K_i$ for some compact $K_i\subset M_i$. 
\end{Def}

Let $S^{d-d'}$ be the $d-d'$ dimensional unit sphere. For $d> d'\geq 1$, $\R^d \# (\R^{d'}\times S^{d-d'})$ is not a space with varying dimension but, by considering the ball to have large radius, it looks similar to a space with varying dimension. 
Furthermore, our large time heat kernel estimates for BMVD are, up to the distances with which the results are stated, of the same form as those for Brownian motion on $\R^d \# (\R^{d'}\times S^{d-d'})$ given in \cite{GSc,GIS}. In fact, in order to prove Theorem \ref{d'=1}, \ref{d'=2,3}, we borrow some techniques from \cite{GSc}.

Finally, we give some remark about the approach by Chen and Lou (\cite{CL}). For large time, they estimated $p(t,a^*,a^*)$ by using the estimate of $\mathbb{P}_x(\sigma_{a^*}\in ds)$ and the Markov property $p(t,a^*,a^*)=\int p(t/2,a^*,x)^2 dm_p(x)$, and obtained desired off-diagonal estimates. By careful calculations, their method also works for general dimensions. However, our method gives relations between the behaviour of BMVD and that of Brownian motion on the connected sum of manifolds studied by \cite{GSc, GIS}, which is of independent interest. So we take this approach.\\

{\bf Acknowledgements\ } I would like to thank Professor Takashi Kumagai, my supervisor, for helpful discussions and for carefully reading this paper, and Professor Ryoki Fukushima for useful comments about Brownian motion with darning. I also thank Professor Laurent Saloff-Coste for giving me important advice about the relationship between the problem studied here and the estimation of heat kernels on the connected sum of manifolds, Professor David A. Croydon for checking the introduction of this paper. After the manuscript was written, Professor Zhen-Qing Chen pointed me out that the approach in \cite{CL} should work for general dimensions as well, and it was indeed true. He also gave me valuable comments including one concerning the proof of Proposition 5.1. I would deeply thank him for the comments.

\section{Preliminary}
Throughout the paper, we fix $\varepsilon, \varepsilon ' ,p>0$. In this section, we first prove the existence and the uniqueness of BMVD. We  then show the existence and some properties of the heat kernel for BMVD. We also prove the space with varying dimension fails to the volume doubling property and we give some lemmas that will be used in Section \ref{lsec1}-\ref{lsec2}.

\begin{Thm}\label{BMVDexuni} For $d\geq d'\geq 1$, $\varepsilon, \varepsilon'>0$ and $p>0$, BMVD with parameters $(\varepsilon, \varepsilon',p)$ on $\R_{\varepsilon}^d \cup \R_{\varepsilon'}^{d'}\cup \{a^*\}$ exists and is unique in law. Furthermore, its associated Dirichlet form $(\mathcal{E},\mathcal{F})$ on $L^2(\R_{\varepsilon}^d \cup \R_{\varepsilon'}^{d'}\cup \{a^*\};m_p)$ is given by 
\begin{eqnarray*} \mathcal{F}:=\left \{f\in L^2(\R_{\varepsilon}^d \cup \R_{\varepsilon'}^{d'}\cup \{a^*\};m_p)\  \middle|\  \begin{split} f|_{\R_{\varepsilon}^d}\in H^1(\R_{\varepsilon}^d), f|_{\R_{\varepsilon'}^{d'}}\in H^1(\R_{\varepsilon'}^{d'})\\
f(x)=f(a^*)\  {\rm q.e.\  on}\  \partial \R_{\varepsilon}^d \cup \partial \R_{\varepsilon'}^{d'} \end{split} \right \}, \\
\mathcal{E}(f,g):=\frac{1}{2}\int_{\R_{\varepsilon}^d}\nabla f \cdot \nabla g \ dm_p +\frac{1}{2}\int_{\R_{\varepsilon'}^{d'}}\nabla f \cdot \nabla g \ dm_p \ {\rm for}\  f,g\in \mathcal{F} . \end{eqnarray*}
 \end{Thm}

\begin{proof}
The proof is the same as that of {\cite[Theorem 2.2]{CL}}.
\end{proof}

\begin{Prop}\label{Nash} There exists a heat kernel $p(t,x,y)$ with respect to $m_p$ which is continuous for each $t>0$. Moreover, for all $t>0$, it holds that $p(t,a^*,a^*)\lesssim t^{-{d}/{2}}\vee t^{-{d'}/{2}}$. \end{Prop}
\begin{proof}
$\| \cdot \| _{L^i}$ denotes $L^i$-norm with respect to $m_p$. Since $\R_{\varepsilon}^d$ and $\R_{\varepsilon'}^{d'}$ have smooth boundaries, for all $f\in \mathcal{F} \cap L^1(\R_{\varepsilon}^d\cup \R_{\varepsilon'}^{d'}\cup \{a^*\})$, by classical Nash's inequality, there is $C>0$ such that 
	$$\|f|_{\R_{\varepsilon}^d}\|^{1+{2}/{d}}_{L^2}\leq C \|f|_{\R_{\varepsilon}^d}\|^{{2}/{d}}_{L^1}\cdot \| \nabla f|_{\R_{\varepsilon}^d}\|_{L^2},$$
	$$\|f|_{\R_{\varepsilon}^{d'}}\|^{1+{2}/{d'}}_{L^2}\leq C \|f|_{\R_{\varepsilon}^{d'}}\|^{{2}/{d'}}_{L^1}\cdot \| \nabla f|_{\R_{\varepsilon}^{d'}}\|_{L^2}.$$
	Then, for all $f\in \mathcal{F}$, we have
	$$\|f\|_{L^2}^2\leq C\left( \mathcal{E}(f,f)^{{d}/{d+2}}\|f\|_{L^1}^{{4}/{d+2}}+\mathcal{E}(f,f)^{{d'}/{d'+2}}\|f\|_{L^1}^{{4}/{d'+2}} \right) .$$
	By {\cite[Corollary 2.12]{CKS}}, the heat kernel $p(t,x,y)$ with respect to $m_p$ exists and the desired inequarity holds for a.e. $x,y$, so it is sufficient to prove the continuity of $p(t,\cdot,\cdot)$. By Definition \ref{BMVD_def} (i), $p(t,\cdot,\cdot)$ is continuous on $(\R_{\varepsilon}^d\cup \R_{\varepsilon'}^{d'})\times (\R_{\varepsilon}^d\cup \R_{\varepsilon'}^{d'})$. For fixed $t,y$, $p(t,x,y)=\int p(t/2,x,z)p(t/2,z,y)dm_p(z)=P_{t/2}p(t/2,\cdot,y)$ is quasi-continuous ({\cite[Proposition 3.1.9]{CF}}) and, since $a^*$ is nonpolar for $X$, $p(t,\cdot, y)$ is continuous. By the symmetry, $p(t, \cdot,\cdot)$ is continuous.
\end{proof}

In this paper, for $x$ and $r>0$, we define $B(x;r):=\{y\in \R_{\varepsilon}^d \cup \R_{\varepsilon'}^{d'}\cup \{a^*\}\ |\ \rho(x,y)<r \}$.

\begin{Prop}\label{VD}
For $d>d'\geq 1$, the volume doubling property fails on $\R_{\varepsilon}^d \cup \R_{\varepsilon'}^{d'}\cup \{a^*\}$ for $m_p$.
\end{Prop}
\begin{proof}
For $r>0$, we take $x\in \R_{\varepsilon}^d$ with $|x|=r+\varepsilon$, see Figure \ref{vdfig}, then we have
$$m_p\left(B(x;r)\right)=\frac{\pi ^{d/2}}{\Gamma (d/2+1)}r^d \ \ {\rm and}$$ 
$$m_p\left(B(x;2r)\right)\geq \frac{\pi ^{d/2}r^d}{\Gamma (d/2+1)}+\frac{p\pi ^{d'/2}\left((r+\varepsilon')^{d'}-\varepsilon'^{d'} \right)}{\Gamma(d'/2+1)}\geq c(r^d+r^{d'}).$$
Now, if there exists $C>0$ such that $m_p\left( B(x;2r) \right)\leq C\ m_p\left( B(x;r) \right)$ for all $x$, then we obtain  $r^d+r^{d'}\leq  cr^d$, so $1+r^{d'-d}\leq  c$. $1+r^{d'-d} \to \infty$ as $r\to 0$ and this is a contradiction.\end{proof}
\begin{figure}[h]
\hspace{10mm}
\tdplotsetmaincoords{65}{10}
\begin{tikzpicture}[tdplot_main_coords]
 \draw[->,>=stealth,very thick] (-3,0,0)--(3,0,0)node[above]{};
 \draw[->,>=stealth,very thick] (0,-4,0)--(0,4,0)node[right]{};
 \draw[->,>=stealth,very thick] (0,0,-2)--(0,0,3)node[right]{};
 \draw (-2,1.0,1.8)--(-2,1.0,2.4);
 \draw (-2,1,1.8)--(-2.6,1,1.75);
 \draw (-2.3,1,2.1)node{$\R^{d}_{\varepsilon}$};

 \node[circle,shading=ball, outer color=red, inner color=pink, minimum width=8mm][label=below left:$a^*$] (ball) at (0,0) {};
 
 \shade[ball color = blue!80, opacity = 0.4] (1.2,0) circle (0.8cm);
 \shade[ball color = gray!40, opacity = 0.4] (1.2,0) circle (1.6cm);
 \shade[ball color = gray!40, opacity = 0.4] (0,0,0) circle (1cm);

\draw[] (1.2,0,0) -- node[above]{$r$} (1.8,0.3,0.4);
\draw[] (1.2,0,0) -- node[below]{$2r$} (2.7,-0.5,-0.5);
\draw[] (-0.33,0.1,0.2) -- node[above]{$r$} (-0.9,0.1,0.4);

\fill (1.2,0,0) circle (2pt) coordinate (x) circle node [above] {$x$};

 \end{tikzpicture}
 \begin{tikzpicture}
 \draw[->,>=stealth,very thick] (-2,0)--(2,0)node[above]{};
 \draw[->,>=stealth,very thick] (0,-1.5)--(0,2)node[right]{};
 \draw (1.2,2.0)--(1.2,1.4);
 \draw (1.2,1.4)--(1.8,1.4);
 \draw (1.6,1.7)node{$\R^{d'}_{\varepsilon'}$};

 \draw[fill=red!80](0,0)circle(0.4);
 \draw[ball color = gray!40, opacity = 0.4] (0,0)circle(1);
\draw[] (0.2828,0.2828) -- node[above]{$r$} (0.7071,0.70710);

 \draw (-0.2,0)node[above left]{$a^*$} ;

\end{tikzpicture}
 
\caption{$B(x;r)$ and $B(x;2r)$}
\label{vdfig}
 \end{figure}
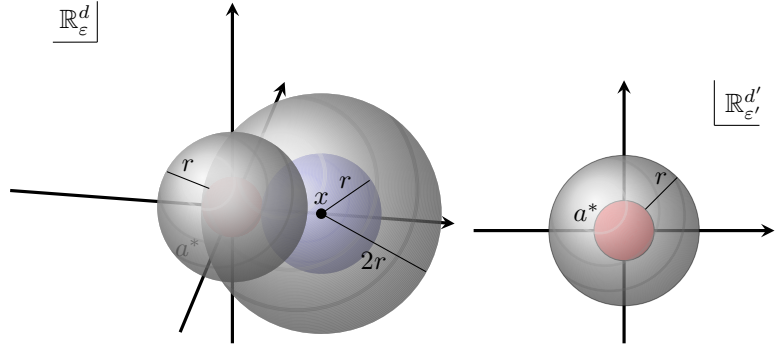

Let $p_{\R_{\varepsilon}^d}(t,x,y)$ be the transition density of the part process of BMVD killed upon exiting $\R_{\varepsilon}^d$. According to \cite{Z}, the following proposition holds.
\begin{Prop}\label{part kernel}
Let $d\geq 2$. For $x,y\in \R_{\varepsilon}^d$ and $t>0$, it holds that\begin{equation}p_{\R_{\varepsilon}^d}(t,x,y)\asymp \frac{1}{t^{{d}/{2}}}\left(1\wedge \frac{|x|_{\rho}}{\sqrt{t}\wedge 1} \right)\left(1\wedge \frac{|y|_{\rho}}{\sqrt{t}\wedge 1} \right)\ee. \label{eq:part kernel}\end{equation}
\end{Prop}

Let $\overline{p}_{\R_{\varepsilon}^d}(t,x,y):=\int_0^t p(t-s,a^*,y)\mathbb{P}_x(\sigma_{a^*}\in ds)$ for $x,y\in \R_{\varepsilon}^d$. In order to estimate $\overline{p}_{\R_{\varepsilon}^d}(t,x,y)$, we prepare some lemmas for $\sigma_{a^*}$.
According to {\cite[Theorem 3]{BMR}}, the following two lemmas hold when $\varepsilon =1$. By the scaling, they hold for every $\varepsilon >0$.
\begin{Lem}\label{hitting3}
For $d\geq 3$ and $x\in \R_{\varepsilon}^d$, it holds that$$\mathbb{P}(\sigma_{a^*}\in ds)\asymp \frac{|x|_{\rho}}{|x|}\frac{e^{-{|x|_{\rho}^2}/{s}}}{s^{{d}/{2}}+s^{{3}/{2}}|x|^{{(d-3)}{/2}}}ds.$$
\end{Lem}		

\begin{Lem}\label{hitting2}
For $d=2$ and $x\in \R_{\varepsilon}^2$, it holds that$$\mathbb{P}(\sigma_{a^*}\in ds)\asymp \frac{|x|_{\rho}}{|x|}\frac{1+\log{|x|}}{\left(1+\log{\left( 1+s/|x| \right)}\right)\left(1+\log{(s+|x|)} \right)}\frac{(|x|+s)^{{1}/{2}}}{s^{3/2}}e^{-|x|_{\rho}/s}ds.$$
\end{Lem}	

We will use the following elementary estimate.
\begin{Lem}\label{elem3}Let $d\geq 3$. Then for $t\geq 1$ and $x\in \R_{\varepsilon}^d$, we have
	\begin{equation}\frac{\ex}{t^{{d}/{2}}+t^{{3}/{2}}|x|^{{(d-3)}/{2}}} \gtrsim \frac{\ex}{t^{{d}/{2}}}.\label{eq:elem31}\end{equation}
	
\end{Lem}
\begin{proof}
When $|x|\leq {\sqrt{t}}/{2}$,  (\ref{eq:elem31}) follows from $t^{{d}/{2}}+t^{{3}/{2}}|x|^{{(d-3)}/{2}}\lesssim t^{{d}/{2}}+t^{{3}/{2}}t^{{(d-3)}/{4}} \lesssim t^{{d}/{2}}$.\\
When $|x|> {\sqrt{t}}/{2}$,
\begin{eqnarray*}\frac{e^{-c|x|^2_{\rho}/t}}{t^{{d}/{2}}+t^{{3}/{2}}|x|^{{(d-3)}/{2}}} \geq \frac{e^{-(c+1)|x|^2_{\rho}/t}\left( \frac{|x|^2_{\rho}}{t} \right)^{{(d-3)}/{4}}}{t^{{d}/{2}}+t^{{3}/{2}}|x|^{{(d-3)}/{2}}}
 \gtrsim \frac{\ex}{t^{{d}/{2}}+t^{{(d+3)}/{4}}}\gtrsim \frac{\ex}{t^{{d}/{2}}}.\end{eqnarray*}
\end{proof}

In the next two lemmas, we obtain the estimates of hitting distribution.
\begin{Lem}\label{GSh3}
Let $d\geq 3$. Then for $x\in \R_{\varepsilon}^d$ and $t>1$, we have
\begin{equation}\mathbb{P}_x(\sigma_{a^*}\leq t)\asymp \frac{1}{|x|^{d-2}}\ex. \label{eq:GS3} \end{equation}
\end{Lem}
\begin{proof}When $|x|_{\rho}\geq 1$, (\ref{eq:GS3}) follows from {\cite[Theorem 4.4 (1)]{GSh}}.\\
When $|x|_{\rho}< 1$, $\mathbb{P}_x(\sigma_{a^*}\leq t)\leq 1$ and there is some $C>0$ with $\mathbb{P}_x(\sigma_{a^*}\leq t)\geq \mathbb{P}_x(\sigma_{a^*}\leq 1)\geq C$, so (\ref{eq:GS3}) holds.
\end{proof}

\begin{Lem}\label{GSh2}
	Let $d=2$. Then for $x\in \R_{\varepsilon}^2$ with $|x|_{\rho}\geq1$, we have
	$$( {\rm i})\ \mathbb{P}_x(\sigma_{a^*}\leq t)\asymp \frac{1}{\log{|x|}}\ex \ {\rm for}\  0<t<2|x|^2, $$
	$$({\rm ii})\ \mathbb{P}_x(\sigma_{a^*}\leq t)\asymp 1-\frac{\log{|x|}}{\log{\sqrt{t}}}\  {\rm for} \ 2|x|^2\leq t. \hspace{12mm}$$
	\end{Lem}
\begin{proof}
See {\cite[Theorem 4.11]{GSh}}.
\end{proof}

The next lemma gives the relations between $\e$, $\exy$ and $\ee$ for large time.
\begin{Lem}\label{rholem}
Let $T>0$ be fixed and $d\geq 1$. For $T\leq t$ and $x,y\in \R_{\varepsilon}^d$, we have
$$({\rm i})\e \asymp \ee \gtrsim \exy \ {\rm if}\  |x|_{\rho}\vee |y|_{\rho} > 1,$$
$$({\rm ii})\e \asymp \ee \asymp \exy \  {\rm if}\  |x|_{\rho}\vee |y|_{\rho} \leq 1,$$
$$({\rm iii})\e \asymp \ee \asymp \exy \hspace{27mm}$$ \hspace{57mm} ${\rm if}\  |x|_{\rho}> 1>b \geq |y|_{\rho} \ {\rm for \ some}\  b.$

\end{Lem}

\begin{proof}
	$($i$)$ When\  $|x|_{\rho}\vee |y|_{\rho} > 1$, we may assume $|x|_{\rho}>1$ without loss of generality. If $\rho(x,y)=|x-y|$, there is nothing to prove.
	If $\rho(x,y)=|x|_{\rho}+|y|_{\rho}$, then we have
	$$\rho(x,y)\leq |x-y|\leq |x|+|y|=|x|_{\rho}+|y|_{\rho}+2\varepsilon \leq (2\varepsilon +1)(|x|_{\rho}+|y|_{\rho})=(2\varepsilon +1)\rho(x,y).$$
	Hence, the desired estimate holds.\\
	$($ii$)$ When\  $|x|_{\rho}\vee |y|_{\rho} \leq 1$, we have
	$$\frac{(|x|_{\rho}+\!|y|_{\rho})^2}{t} \leq \frac{4}{T}\leq \frac{4}{T}+\frac{|x\!-\!y|^2}{t}\leq \frac{4}{T}+\frac{(|x|\!+\!|y|)^2}{t} \leq \frac{4}{T}+\frac{2(2\varepsilon)^2}{T}+\frac{2(|x|_{\rho}+\!|y|_{\rho})^2}{t}.$$
	Hence, desired estimate holds.\\
	$($iii$)$ When\  $|x|_{\rho} \geq 1>b \geq |y|_{\rho}$, we have
	\begin{eqnarray*}\frac{(|x|_{\rho}+|y|_{\rho})^2}{t}-2\frac{|x-y|^2}{t}=\frac{(|x|+|y|-2\varepsilon)^2-2|x-y|^2}{t}\hspace{20mm}\\
	\leq \frac{(|x-y|+|y|+|y|-2\varepsilon)^2-2|x-y|^2}{t}=\frac{(2|y|_{\rho}+|x-y|)^2-2|x-y|^2}{t}\\
\leq  \frac{2(2|y|_{\rho})^2+2|x-y|^2-2|x-y|^2}{t}
=\frac{8|y|^2_{\rho}}{t} \leq \frac{8b^2}{T}. \hspace{32mm}\end{eqnarray*}
Hence, we have $\ee \lesssim e^{-(|x|_{\rho}+|y|_{\rho})^2/2t}e^{4b^2/T}$ and by $($i$)$, the desired estimate holds.
\end{proof}

\section{Small time estimate} \label{smallsec}
In this section, we prove Theorem \ref{small} in the same way as {\cite[section 4]{CL}}. 
First, we define \begin{eqnarray}u(x):=\left \{ \begin{split} -&|x|_{\rho}& :x\in \R_{\varepsilon'}^{d'}\\ &|x|_{\rho}& :x\in \R_{\varepsilon}^d \end{split}\right. \end{eqnarray} and $Y_t:=u(X_t)$. Then $u\in \mathcal{F}^{loc}$, where $\mathcal{F}^{loc}$ denotes the local Dirichlet space of $(\mathcal{E},\mathcal{F})$. We will prove that the heat kernel for $Y$ enjoys 1-dimensional Gaussian estimates. Combining this with the fact that $\overline{p}_{\R_{\varepsilon}^d}(t,x,y)$ (resp. $p(t,x,y)$) depends only on $|x|_{\rho}$ and $|y|_{\rho}$ for $x,y \in \R_{\varepsilon}^d$ (resp. $x\in \R_{\varepsilon}^d, y\in \R_{\varepsilon'}^{d'}$), we prove Theorem \ref{small}.

We first derive the stochastic differential equation that $Y$ satisfies, and then use it to obtain Gaussian heat kernel estimates of $Y$.

 \begin{Prop}
 	\label{dY_t}
 	\begin{eqnarray*}\label{eq:dY}dY_t=dB_t+\frac{(d-1){\bf 1}_{\{Y_t>0\}}}{2(Y_t+\varepsilon)}dt+\frac{p(d'-1){\bf 1}_{\{Y_t<0\}}}{2(Y_t-\varepsilon')}dt+\frac{|\partial \R_{\varepsilon}^d|-p|\partial \R_{\varepsilon'}^{d'}|}{|\partial \R_{\varepsilon}^d|+p|\partial \R_{\varepsilon'}^{d'}|}d\hat{L}_t^0(Y),\end{eqnarray*}
 where $|\cdot|$ is the Lebesgue measure, $B$ is one-dimensional Brownian motion and $\hat{L}^0(Y)$ is symmetric semimartingale local time of $Y$ at $0$. Here, for convenience, we define $|\partial \R_{\varepsilon'}^{d'}|=1$ when $d'=1$.
\end{Prop}
\begin{proof}
	By the Fukushima decomposition ({\cite[Chapter 4]{CF}}), there exist local martingale additive functional $M^{[u]}$ and continuous additive functional locally having zero energy $N^{[u]}$ such that $Y_t-Y_0=M_t^{[u]}+N^{[u]}$, $\mathbb{P}_x$-a.s. for q.e. $x\in \R_{\varepsilon}^d \cup \R_{\varepsilon'}^{d'}\cup\{a^*\}$.
	For any $\psi \in C_c^{\infty}(\R_{\varepsilon}^d \cup \R_{\varepsilon'}^{d'}\cup\{a^*\})$, it holds that
	\begin{eqnarray}\mathcal{E}(u,\psi)&=&\frac{1}{2}\int_{\R_{\varepsilon}^d}\nabla |x|_{\rho}\cdot \nabla{\psi} dx+\frac{p}{2}\int_{\R_{\varepsilon'}^{d'}}\nabla (-|x|_{\rho})\cdot \nabla{\psi} dx\\
	\nonumber&=&-\frac{1}{2}\int_{\R_{\varepsilon}^d}\frac{d-1}{|x|}\psi dx+\frac{1}{2}\int_{\partial \R_{\varepsilon}^d}\psi (a^*)\frac{\partial |x|_{\rho}}{\partial {\bf n}}\sigma(dx)\\
	\nonumber&&+\frac{p}{2}\int_{\R_{\varepsilon'}^{d'}}\frac{d'-1}{|x|}\psi dx-\frac{p}{2}\int_{\partial \R_{\varepsilon'}^{d'}}\psi (a^*)\frac{\partial |x|_{\rho}}{\partial {\bf n}}\sigma(dx)\\
	\nonumber&=&-\frac{1}{2}\int_{\R_{\varepsilon}^d}\frac{d-1}{|x|}\psi dx+\frac{p}{2}\int_{\R_{\varepsilon'}^{d'}}\frac{d'-1}{|x|}\psi dx-\frac{1}{2}(|\partial \R_{\varepsilon}^d|-p|\partial \R_{\varepsilon'}^{d'}|)\psi(a^*)\\
	\nonumber&=&-\int_{\R_{\varepsilon}^d\cup\R_{\varepsilon'}^{d'}\cup \{a^*\}} \psi(x)\nu(dx), \label{eq:Fdecom}\end{eqnarray}
where {\bf n} is the outward normal vector of the surface $\partial \R_{\varepsilon}^d \cup \partial \R_{\varepsilon'}^{d'}$, $\sigma$ is the surface measure on $\partial \R_{\varepsilon}^d \cup \partial \R_{\varepsilon'}^{d'}$, and
 $$\nu(dx):=\frac{d-1}{2|x|}{\bf 1}_{\R_{\varepsilon}^d}(x)dx-\frac{p(d'-1)}{2|x|}{\bf 1}_{\R_{\varepsilon'}^{d'}}(x)dx+\frac{(|\partial \R_{\varepsilon}^d|-p|\partial \R_{\varepsilon'}^{d'}|)}{2}\delta_{\{a^*\}}.$$
By {\cite[Theorem 5.5.5]{FOT}}, it holds that
 $$dN_{t}^{[u]}=\frac{(d-1){\bf 1}_{\{X_t\in \R_{\varepsilon}^d\}}}{2(u(X_t)+\varepsilon)}dt-\frac{p(d'-1){\bf 1}_{\{X_t\in \R_{\varepsilon'}^{d'}\}}}{2(-u(X_t)+\varepsilon')}dt+(|\partial \R_{\varepsilon}^d|-p|\partial \R_{\varepsilon'}^{d'}|)dL_t^0(X).$$
Here, $L_t^0(X)$ is the positive continuous additive functional of $X$ whose Revuz measure is $\frac{1}{2}\delta_{\{a^*\}}$.\\
Let $u_n:= (-n\vee u)\wedge n$, then $u_n\in \mathcal{F}$. By {\cite[Theorem 4.3.11]{CF}} and strongly locality of $(\mathcal{E},\mathcal{F})$, for any $\varphi \in \mathcal{F}\cap C_c(\R_{\varepsilon}^d\cup\R_{\varepsilon'}^{d'}\cup \{a^*\})$, we have $\int \varphi d\mu _{\langle u_n \rangle}=2\mathcal{E}(u_n\varphi,u_n)-\mathcal{E}(u_n^2,\varphi)=\int\varphi |\nabla u_n|^2dm_p$. Here, $d\mu _{\langle u_n \rangle}$ is the Revuz measure corresponding to $\langle M^{[u_n]}\rangle$. Then we obtain $d\mu _{\langle u_n \rangle}=|\nabla u_n|^2dm_p={\bf 1}_{\{|x|_{\rho}\leq n\}}dm_p$. It yields $d\mu _{\langle u \rangle}=dm_p$. By {\cite[Theorem 4.1.8]{CF}}, $\langle M^{[u]}\rangle _t=t$ and $B_t:=M_t^{[u]}$ is one-dimensional Brownian motion. Thus it holds that
\begin{eqnarray}\nonumber dY_t=dB_t+\frac{(d-1){\bf 1}_{\{Y_t>0\}}}{2(Y_t+\varepsilon)}dt&+&\frac{p(d'-1){\bf 1}_{\{Y_t<0\}}}{2(Y_t-\varepsilon')}dt\\
&+&(|\partial \R_{\varepsilon}^d|-p|\partial \R_{\varepsilon'}^{d'}|)dL_t^0(X) .\label{eq:dYdL}\end{eqnarray}
Next, we show $d\hat{L}_t^0(Y)=(|\partial \R_{\varepsilon}^d|+p|\partial \R_{\varepsilon'}^{d'}|)dL_t^0(X)$.\\
Let $v(x):=|x|_{\rho}$, so $|Y_t|=v(X_t)$ holds. Then, by the similar computation as above, for one-dimensional Brownian motion $\tilde{B}$, we have
$$d|Y_t|=d\tilde{B}_t+\frac{(d-1){\bf 1}_{\{Y_t>0\}}}{2(Y_t+\varepsilon)}dt-\frac{p(d'-1){\bf 1}_{\{Y_t<0\}}}{2(Y_t-\varepsilon')}dt+(|\partial \R_{\varepsilon}^d|+p|\partial \R_{\varepsilon'}^{d'}|)dL_t^0(X).$$
While, by Tanaka's formula and (\ref{eq:dYdL}), we have \begin{eqnarray}d|Y_t|&=&{\rm sign}(Y_t)dY_t+dL_t^0(Y)\label{eq:Tanaka} \\
\nonumber &=&{\rm sign}(Y_t)dB_t+\frac{(d-1){\bf 1}_{\{Y_t>0\}}}{2(Y_t+\varepsilon)}dt-\frac{p(d'-1){\bf 1}_{\{Y_t<0\}}}{2(Y_t-\varepsilon')}dt\\
\nonumber &&-(|\partial \R_{\varepsilon}^d|-p|\partial \R_{\varepsilon'}^{d'}|)dL_t^0(X)+dL_t^0(Y),\end{eqnarray}
where ${\rm sign}(x):={\bf 1}_{\{x>0\}}-{\bf 1}_{\{x\leq0\}}$. By the uniqueness of the decomposition of a continuous semi-martingale to a continuous local martingale and a continuous bounded variation process, we have $dL_t^0(Y)=2|\partial \R_{\varepsilon}^d|dL_t^0(X)$.\\
By the similar computation as above for $-Y$ and $|Y_t|=|-Y_t|$, it holds that $dL_t^0(-Y)=dL_t^0(Y)-2(|\partial \R_{\varepsilon}^d|-p|\partial \R_{\varepsilon'}^{d'}|)dL_t^0(X)$. Then we have 
\begin{eqnarray}d\hat{L}^0_t(Y)=\frac{dL_t^0(Y)+dL_t^0(-Y)}{2}=(|\partial \R_{\varepsilon}^d|+p|\partial \R_{\varepsilon'}^{d'}|)dL_t^0(X) .\label{eq:L^} \end{eqnarray}
By (\ref{eq:dYdL}) and (\ref{eq:L^}), the desired SDE follows.
\end{proof}

\begin{Prop}
\label{Yestimate}
	$Y$ has a jointly continuous density function $p^{(Y)}(t,x,y)$ with respect to the Lebesgue measure on $\R$. Furthermore, for any $T\geq 1$, $p^{(Y)}(t,x,y)\asymp \frac{1}{\sqrt{t}}\ee$ for $(t,x,y)\in (0,T]\times \R \times \R$.
\end{Prop}

\begin{proof}
	This follows from the proof of {\cite[Proposition 4.4]{CL}}.
\end{proof}

In the following propositions, we prove Theorem \ref{small}.
\begin{Prop}[Theorem \ref{small}(iii)]\label{smallprop3}
Fix $T\geq 1$. Then it holds that $$p(t,x,y)\asymp \frac{1}{\sqrt{t}}\e \  {\rm for} \ t\in(0,T], x\in \R_{\varepsilon'}^{d'}\cup \{a^*\}, y\in \R_{\varepsilon}^{d}\cup \{a^*\}.$$
\end{Prop}
\begin{proof}
	Since BMVD hits $a^*$, we have 
	\begin{eqnarray*}p(t,x,y)&=&\int _0^t p(t-s,a^*,x)\mathbb{P}_y(\sigma_{a^*}\in ds)\\
	&=&\int _0^t\int_0^{t-s} p(t-s-w,a^*,a^*)\mathbb{P}_x(\sigma_{a^*}\in dw)\mathbb{P}_y(\sigma_{a^*}\in ds). \end{eqnarray*}
	Thus $(x,y)\mapsto p(t,x,y)$ depends only on $|x|_{\rho}$ and $|y|_{\rho}.$ For $a>b>0$, we have \begin{eqnarray}\int_a^bp^{(Y)}(t,-|x|_{\rho},|y|_{\rho})d|y|_{\rho}&=&\mathbb{P}_{-|x|_{\rho}}(a\leq Y_t\leq b)\label{eq:siii1} \\
	\nonumber&=&\mathbb{P}_{x}(X_t\in\R_{\varepsilon}^d,a\leq |X_t|_{\rho}\leq b)\\
	\nonumber&=&\int_{\{y\in \R_{\varepsilon}^d, a\leq |y|_{\rho}\leq b\}}p(t,x,y)m_p(dy)\\
	\nonumber&=&\int_a^b |\partial B(0;|y|_{\rho}+\varepsilon)|p(t,x,y)d|y|_{\rho} .\end{eqnarray}
	Thus, we have $p^{(Y)}(t,-|x|_{\rho},|y|_{\rho})d|y|_{\rho}=|\partial B(0;|y|)|p(t,x,y)\asymp |y|^{d-1}p(t,x,y)$.\\
	By Proposition \ref{Yestimate}, it holds that \begin{eqnarray}p(t,x,y)\asymp \frac{1}{|y|^{d-1}}\frac{1}{\sqrt{t}}e^{-\left(-|x|_{\rho}-|y|_{\rho}\right)^2/t}=\frac{1}{|y|^{d-1}\sqrt{t}}\e. \label{eq:siii2}\end{eqnarray}
	Since $\varepsilon \leq |y|$ and (\ref{eq:siii2}), we have $$p(t,x,y)\asymp \frac{1}{|y|^{d-1}\sqrt{t}}\e \lesssim \frac{1}{\sqrt{t}}\e.$$\\
	Moreover, if $|y|_{\rho}\leq 1$ we have $p(t,x,y)\gtrsim \frac{1}{\sqrt{t}}\e$ and if $|y|_{\rho}>1$ we have
	 \begin{eqnarray*}p(t,x,y) &\asymp &\frac{1}{|y|^{d-1}\sqrt{t}}\e \geq \frac{1}{|y|^{d-1}\sqrt{t}}\left(\frac{t}{T}\right)^{{(d-1)}/{2}}\e \\
	&\gtrsim &\frac{1}{\rho(x,y)^{d-1}\sqrt{t}}\left(\frac{t}{T}\right)^{{(d-1)}/{2}}e^{-c{\rho(x,y)^2}/{t}} \gtrsim \frac{1}{\sqrt{t}}e^{-(c+1){\rho(x,y)^2}/{t}}.  \end{eqnarray*}\end{proof}

\begin{Prop}[Theorem \ref{small}(ii)]\label{smallprop2}Fix $T\geq 1$, then for all $t\leq T, x,y\in \R_{\varepsilon}^d$, it holds that
$$p(t,x,y)\asymp \frac{\e}{\sqrt{t}}+\frac{\ee}{t^{{d}/{2}}}\left( 1\wedge \frac{|x|_{\rho}}{\sqrt{t}}\right)\left( 1\wedge \frac{|y|_{\rho}}{\sqrt{t}}\right) {\rm if}\ |x|_{\rho}\vee |y|_{\rho} \leq 1,$$
$$p(t,x,y)\asymp \frac{1}{t^{{d}/{2}}}\e \ {\rm if}\  |x|_{\rho}\vee |y|_{\rho} > 1.\hspace{50mm}$$
\end{Prop}

\begin{proof}
	When $d=1$, the statement holds from Proposition \ref{Yestimate}, so we assume $d\geq 2$.\\
	For $x,y\in \R_{\varepsilon}^d$, it holds that $p(t,x,y)=p_{\R_{\varepsilon}^d}(t,x,y)+\overline{p}_{\R_{\varepsilon}^d}(t,x,y)$.
	Since $\overline{p}_{\R_{\varepsilon}^d}(t,x,y)$ depends only on $|x|_{\rho}$ and $|y|_{\rho}$, for $0<a<b$, we have
	\begin{eqnarray}\mathbb{P}_x(\sigma_{a^*}<t, X_t\in \R_{\varepsilon}^d, a\leq |X_t|_{\rho}\leq b)=&\int _{\{a\leq|y|_{\rho}\leq b\}}\overline{p}_{\R_{\varepsilon}^d}(t,x,y)m_p(dy)\label{eq:hita} \\
	\nonumber \asymp &\int_a^b (|y|_{\rho}+\varepsilon)^{d-1}\overline{p}_{\R_{\varepsilon}^d}(t,x,y)d|y|_{\rho}.  \end{eqnarray}
The left hand side of (\ref{eq:hita}) is equal to $$\mathbb{P}_{|x|_{\rho}}^{(Y)}\left(\sigma _0<t, Y_t>0,a\leq Y_t\leq b\right)=\int_a^b\int_0^tp^{(Y)}(t-s,0,|y|_{\rho})\mathbb{P}^{(Y)}_{|x|_{\rho}}(\sigma_0\in ds)d|y|_{\rho}.$$
Here, $\mathbb{P}^{(Y)}$ is a probability measure with respect to $Y$. Thus, by using Proposition \ref{Yestimate}, it follows that
\begin{eqnarray}\nonumber (|y|_{\rho}+\varepsilon)^{d-1}\overline{p}_{\R_{\varepsilon}^d}(t,x,y)&\asymp &\int_0^tp^{(Y)}(t-s,0,|y|_{\rho})\mathbb{P}^{(Y)}_{|x|_{\rho}}(\sigma_0\in ds)d|y|_{\rho}\\
\nonumber&\asymp& \int_0^tp^{(Y)}(t-s,0,|y|_{\rho})\mathbb{P}^{(Y)}_{-|x|_{\rho}}(\sigma_0\in ds)d|y|_{\rho}\\
\nonumber&=&p^{(Y)}(t,-|x|_{\rho},|y|_{\rho})\\
&\asymp& \frac{1}{\sqrt{t}}\exy. \label{eq:smallii1}\end{eqnarray}
	{\bf Case1} $|x|_{\rho}\vee |y|_{\rho} \leq 1$: Since $\varepsilon \leq |y|_{\rho}+\varepsilon\leq1+\varepsilon$, we have by (\ref{eq:smallii1}), 
	\begin{equation}\overline{p}_{\R_{\varepsilon}^d}(t,x,y)\asymp \frac{1}{\sqrt{t}}\exy .\label{eq:smalliia}\end{equation}
	If $\rho(x,y)\asymp |x|_{\rho}+|y|_{\rho}$, we obtain $\overline{p}_{\R_{\varepsilon}^d}(t,x,y)\asymp \frac{1}{\sqrt{t}}\e.$\\
	If $\rho(x,y)= |x-y|$ and $|x|_{\rho}\wedge|y|_{\rho} \leq \sqrt{t}$, we may assume $|x|_{\rho}\leq \sqrt{t}$ without loss of generality. Then, it holds that
	\begin{eqnarray}\nonumber \rho(x,y)&\leq &|x|_{\rho}+|y|_{\rho}\leq \sqrt{t} +|y|_{\rho}\leq \sqrt{t}+ |x|+|x-y|-\varepsilon \\
	& =&\sqrt{t}+ |x|_{\rho}+|x-y|\leq 2\sqrt{t}+|x-y|.\label{eq:smalllem1}\end{eqnarray}
	By (\ref{eq:smalllem1}), it holds that $\e \geq \exy \geq e^{-2(2\sqrt{t})^2/t}\e$. Thus, by (\ref{eq:smalliia}), we have $\overline{p}_{\R_{\varepsilon}^d}(t,x,y)\asymp \frac{1}{\sqrt{t}}\e$.\\
	If $\rho(x,y)= |x-y|$ and $|x|_{\rho}\wedge|y|_{\rho} > \sqrt{t}$, by (\ref{eq:smalliia}) and (\ref{eq:part kernel}), we have 
	$$p(t,x,y)\asymp \frac{1}{\sqrt{t}}\exy+\frac{1}{t^{{d}/{2}}}\ee \lesssim \frac{1}{\sqrt{t}}\e+\frac{1}{t^{{d}/{2}}}\ee,
	$$
	\begin{eqnarray*}p(t,x,y)\gtrsim \frac{1}{t^{{d}/{2}}}\ee \gtrsim \frac{1}{\sqrt{t}}\e+\frac{1}{t^{{d}/{2}}}\ee.
	\end{eqnarray*}	
	{\bf Case 2} $|x|_{\rho}\vee |y|_{\rho} > 1$: Without loss of generality, we may assume $|y|_{\rho}> 1$. By (\ref{eq:smallii1}), it holds that
	\begin{eqnarray}\nonumber \overline{p}_{\R_{\varepsilon}^d}(t,x,y)&\asymp& \frac{1}{(|y|_{\rho}+\varepsilon)^{d-1}}\frac{1}{\sqrt{t}}\exy \\
	\nonumber&\geq& \frac{1}{2(\varepsilon+1)(|x|_{\rho}+|y|_{\rho})^{d-1}}\frac{1}{\sqrt{t}}\exy \\
	\nonumber&\gtrsim&\frac{1}{(|x|_{\rho}+|y|_{\rho})^{d-1}}\frac{1}{\sqrt{t}}\left(\frac{(|x|_{\rho}+|y|_{\rho})^2}{t} \right)^{(d-1)/2}\exy\\
	&\asymp& \frac{1}{t^{{d}/{2}}}\exy .\label{eq:smallii2}\end{eqnarray}
	By (\ref{eq:part kernel}) and (\ref{eq:smallii1}), we obtain
	\begin{eqnarray}\nonumber p(t,x,y)&\asymp& \frac{1}{t^{{d}/{2}}}\left(1\wedge \frac{|x|_{\rho}}{\sqrt{t}} \right)\left(1\wedge \frac{|y|_{\rho}}{\sqrt{t}} \right)\ee +\frac{\exy}{\sqrt{t}(|y|_{\rho}+\varepsilon)^{d-1}}\label{eq:smalliib}\\
	&\lesssim&\frac{1}{t^{{d}/{2}}}\left(\ee+\exy \right).
	\end{eqnarray}
	If $|x|_{\rho}\wedge|y|_{\rho} \leq \sqrt{t}$, then we have $p(t,x,y)\asymp \frac{1}{t^{{d}/{2}}}\e$ in the same way as Case1.\\
	If $|x|_{\rho}\wedge|y|_{\rho} > \sqrt{t}$, then we have $p(t,x,y)\asymp \frac{1}{t^{{d}/{2}}}\e$ since $\rho(x,y)=|x-y|\wedge (|x|_{\rho}+|y|_{\rho})$.\\
	This completes the proof.

\end{proof}

\begin{Prop}[Theorem \ref{small}(i)]\label{smallprop1}Fix $T\geq 1$, then for all $t\in(0,T], x,y\in \R_{\varepsilon'}^{d'},$ it holds that
$$p(t,x,y)\asymp \frac{\e}{\sqrt{t}}+\frac{\ee}{t^{{d'}/{2}}}\left( 1\wedge \frac{|x|_{\rho}}{\sqrt{t}}\right)\left( 1\wedge \frac{|y|_{\rho}}{\sqrt{t}}\right)\ {\rm if}\ |x|_{\rho}\vee |y|_{\rho} \leq 1.$$
$$p(t,x,y)\asymp \frac{1}{t^{{d'}/{2}}}\e \ {\rm if}\ |x|_{\rho}\vee |y|_{\rho} > 1. \hspace{50mm}$$

\end{Prop}
\begin{proof}
The proof is the same as that of Proposition \ref{smallprop2}.
\end{proof}
This completes the proof of Theorem \ref{small}.\\

\section{Large time estimate($d'=1$)} \label{lsec1}
In this section, we prove Theorem \ref{d'=1}. Let $d'=1$. When $d=1$, $\R_+\cup \R_+\cup \{a^*\}$ can be identified with $\R$. In this case, BMVD is 1-dimensional Brownian motion, so there is nothing to prove. When $d=2$, it was proved by \cite{CL}.
Hence we consider the case of $d\geq 3$. 
Let $\varepsilon >0$ and $S^{d-1}_{\varepsilon}:=\{x\in \R^d\ ;\ |x|=\varepsilon\}$.
We will prove Theorem \ref{d'=1} by projecting $(\R_+\times S^{d-1}_{\varepsilon})\# \R^d$ to $\R_+ \cup \R_{\varepsilon}^d\cup \{a^*\}$.

The following theorem is a special case of {\cite[Corollary 6.13]{GSc}}. 
\begin{Thm}
	\label{GSc1}
	Let $K$ be central part of $M:=(\R_+\times S^{d-1}_{\varepsilon})\# \R^d$. Let $E_1:=(M\setminus K)\cap(\R_+\times S^{d-1}_{\varepsilon})$, $E_2:=(M\setminus K)\cap \R^d$, and $E_0\subset M$ be a precompact open set having smooth boundary and containing $K$. Then heat kernel $\check{p}(t,x,y)$ of standard Browian motion $\check{X}$ on $M$ satisfies the following estimates for $1\leq t.$
	\begin{enumerate}
	\item For $x,y\in E_1$, $$\check{p}(t,x,y)\asymp \frac{|x|_e|y|_e}{\sqrt{t}(|x|_e+\sqrt{t})(|y|_e+\sqrt{t})}e^{-{d(x,y)^2}/{t}}.$$
	
	\item For $x,y\in E_2$, $$\check{p}(t,x,y)\asymp \frac{1}{t^{{3}/{2}}|x|_e^{d-2}|y|_e^{d-2}}e^{-{(|x|_e+|y|_e)^2}/{t}}+\frac{1}{t^{{d}/{2}}}e^{-{d(x,y)^2}/{t}}.$$
	
 	\item For $x\in E_0\cup E_1,y\in E_0\cup E_2$, $$\check{p}(t,x,y)\asymp \left(\frac{1}{t^{{d}/{2}}}+\frac{|x|_e}{t^{{3}/{2}}|y|_e^{d-2}} \right)e^{-{\rho (x,y)^2}/{t}}.$$
 	\end{enumerate}
 	Here, $d$ is a geodesic distance, and $|x|_e:=\sup_{z\in K}d(x,z)\asymp 1+d(x,K)$.
 \end{Thm}

From now on, we fix $K:=\left(\{0\} \times \{x\in \R^d; |x|<1+\varepsilon \}\right) \cup \left([0,1)\times S_{\varepsilon}^{d-1}\right)$. Then it holds that $M=(\R_+ \cup \{0\})\times S_{\varepsilon}^{d-1}\  \cup \  \R_{\varepsilon}^d$. See Figure \ref{Fig:ends}.

\begin{figure}[h] 
\begin{tikzpicture}
\draw (-2,5) -- (-3,5) --(-5,3) -- (3,3) -- (5,5) -- (2,5);
\draw[] (-2,5) -- (2,5);

\draw  (2,4) ++ (-2,2) circle (0.5cm and 0.2cm);
\draw  (2,4) ++ (-2,0) circle (0.5cm and 0.2cm);
\draw (-0.5,4)--(-0.5,7);
\draw (0.5,4)--(0.5,7);

\draw  (2,4) ++ (-2,0.6) circle (0.5cm and 0.2cm);
\draw  (2,4) ++ (-2,0) circle (1cm and 0.3cm);

\fill[gray,opacity=0.5] 
 (2,4) ++ (-2,0.6) circle (0.5cm and 0.2cm)
 (2,4) ++ (-2,0) circle (1cm and 0.3cm)
 (-0.5,4.6)--(-0.5,4)--(0.5,4)--(0.5,4.6);
 
 \draw (0.8,4.6)node{$K$};
 \draw (4.3,4.8)node{$R^d$};
  \draw (2,4)node{$E_2$};
 \draw (-1,6)node{$E_1$};
 \draw (1.6,6.5)node{$\R_+\times S^{d-1}_{\varepsilon}$};

\end{tikzpicture}

\caption{$M:=(\R_+\times S^{d-1}_{\varepsilon})\# \R^d\ =\ (\R_+ \cup \{0\})\times S_{\varepsilon}^{d-1}\  \cup \  \R_{\varepsilon}^d$}\label{Fig:ends}

\end{figure}
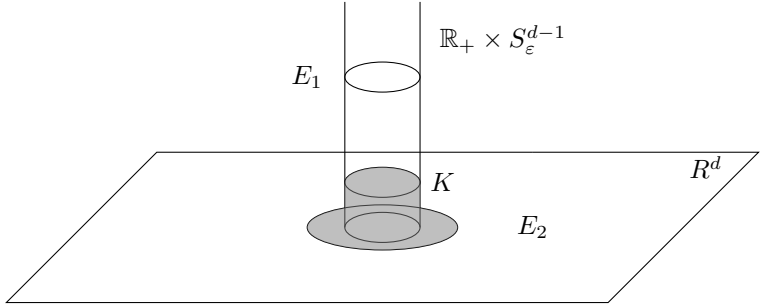

We define $q:=p/|\partial \R_{\varepsilon}^d|$ and $\tilde{m}_q(A):=m^{(d)}(A\cap \R ^d)+q\ m^{(1,d-1)}(A\cap(\R_+\times S^{d-1}_{\varepsilon}))$ for a Borel set $A\subset M$. Here, $m^{(d)}$ and $m^{(1,d-1)}$ are the Lebesgue measures on $\R^d$ and $\R_+\times S^{d-1}_{\varepsilon}$, respectively. 
Then $\tilde{m}_q$-symmetric Brownian motion $\{\tilde{X}_t\}$ on $M$ is a time-changed process of standard Brownian motion $\{\check{X}_t\}$ on $M$ by a positive continuous additive functional having the Revuz measure $\tilde{m}_q$. To be precise, we have $\tilde{X}_t=\check{X}_{\tau_t}$, where $A_t:=\int _0^t ({\bf 1}_{\R^d}+q{\bf 1}_{\R_+\times S^{d-1}_{\varepsilon}})(\check{X}_s)ds$ and $\tau _t:=\{s>0\ |\ A_s>t\}$. Let $\tilde{p}(t,x,y)$ (resp. $\check{p}(t,x,y)$) be the heat kernel of $\{\tilde{X}_t\}$ (resp. $\{\check{X}_t\}$). Since $(1\wedge q)t\leq A_t \leq (1\vee q)t$ and $\frac{t}{1\vee q}\leq \tau _t \leq \frac{t}{1\wedge q}$, we have $\tilde{p}(t,x,y)\asymp \check{p}(t,x,y)$.
 Thus $\tilde{p}(t,x,y)$ satisfies the same estimates as Theorem $\ref{GSc1}$.\\
 
We define $$v(x):= \left\{ \begin{split} -x^{(1)} &: x=(x^{(1)}, x^{(2)})\in (\R_+\cup \{0\})\times S^{d-1}_{\varepsilon},&\\
|x|_{\rho} &: x\in \R_{\varepsilon}^d,& \end{split} \right. $$ 
and $\tilde{Y}_t:=v(\tilde{X}_t).$

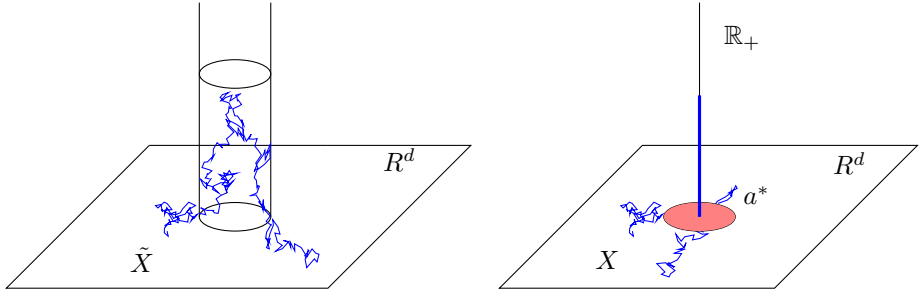
\begin{figure}[h] 
\begin{tikzpicture}[scale=0.95]
\draw (-0.5,5) -- (-2.5,3) -- (2,3) -- (4,5) -- (2,5);
\draw[] (-0.5,5) -- (2,5);

\draw  (2.7,4) ++ (-2,2) circle (0.5cm and 0.2cm);
\draw  (2.7,4) ++ (-2,0) circle (0.5cm and 0.2cm);
\draw (0.2,4)--(0.2,7);
\draw (1.2,4)--(1.2,7);
 \draw (3,4.8)node{$R^d$};
 
\pgfmathsetseed{14348}
\draw[blue] (-0.25,3.8)
\foreach \x in {1,...,52}
{   -- ++(rand*0.1,rand*0.1)
};
\pgfmathsetseed{1}
\draw[blue] (0.2,4.05)
\foreach \x in {1,...,50}
{   -- ++(rand*0.1,-rand*0.1)
};
\pgfmathsetseed{879}
\draw[blue] (0.345,4.53)
\foreach \x in {1,...,55}
{   -- ++(-rand*0.1,-rand*0.1)
};
\pgfmathsetseed{879}
\draw[blue] (0.764,5.72)
\foreach \x in {1,...,45}
{   -- ++(-rand*0.09,rand*0.13)
};

\pgfmathsetseed{1197}
\draw[blue] (0.92,4.3)
\foreach \x in {1,...,16}
{   -- ++(rand*0.11,rand*0.13)
};

\pgfmathsetseed{14787}
\draw[blue] (1.21,3.9)
\foreach \x in {1,...,40}
{   -- ++(-rand*0.11,rand*0.13)
};

\draw (-0.3,3.4)node[left]{$\tilde{X}$};

\pgfmathsetseed{14348}
\draw[blue] (6.26,3.8)
\foreach \x in {1,...,52}
{   -- ++(rand*0.1,rand*0.1)
};

\pgfmathsetseed{14787}
\draw[blue, rotate around={140:(7.6,4.2)}] (7.7,4.3)
\foreach \x in {1,...,13}
{   -- ++(-rand*0.11,rand*0.13)
};
\pgfmathsetseed{14787}
\draw[blue, rotate around={265:(7.63,3.82)}] (7.7,3.45)
\foreach \x in {1,...,40}
{   -- ++(-rand*0.11,rand*0.13)
};

\draw (6.4,5) -- (4.4,3) -- (8.25,3) -- (10.25,5) -- (6.4,5);

\draw  (7.7,4) ++ (-0.5,0) circle (0.5cm and 0.2cm);

\fill[red!50](7.7,4) ++ (-0.5,0) circle (0.5cm and 0.2cm);
\draw (7.2,4)--(7.2,7);
 \draw (9.3,4.75)node{$R^d$};
 \draw (7.8,6.5)node{$\R_+$};
 \draw (8,4.3)node{$a^*$};
\draw[blue,very thick] (7.2,4)--(7.2,5.7); 

\draw (6.21,3.4)node[left]{$X$};

\end{tikzpicture}
\caption{Projection $M$ to $\R_+ \cup \R_{\varepsilon}^d\cup \{a^*\}$, and $\tilde{X}$ and $X$}\label{Fig:proj}
\end{figure}

\begin{Thm}
	\label{proj1}
$\tilde{Y}$ has the same law as $Y$. Here, $Y$ is the signed radial process of $X$ defined in Section \ref{smallsec}.
\end{Thm}

\begin{proof}
By Proposition \ref{dY_t}, it holds that 
\begin{eqnarray}dY_t=dB_t+\frac{(d-1){\bf 1}_{\{Y_t>0\}}}{2(Y_t+\varepsilon)}dt+\frac{|\partial \R_{\varepsilon}^d|-p}{|\partial \R_{\varepsilon}^d|+p}d\hat{L}_t^0(Y), \label{eq:dY1d} \end{eqnarray}
 where $B$ is one-dimensional Brownian motion. We will prove $\tilde{Y}$ also satisfies (\ref{eq:dY1d}). Let $(\mathcal{\tilde{E}},\mathcal{\tilde{F}})$ on $L^2(M;d\tilde{m}_q)$ be the Dirichlet form associated with $\tilde{X}$. Then we have $v\in \mathcal{\tilde{F}}^{loc}$ . By the Fukushima decomposition, there exist local martingale additive functional $M^{[v]}$ and continuous additive functional locally having zero energy $N^{[v]}$ such that $\tilde{Y}_t-\tilde{Y}_0=M_t^{[v]}+N_t^{[v]}$, $\mathbb{P}_x$-a.s. for q.e. $x\in M$.
	For any $\psi \in C_c^{\infty}(M)$, it holds that
	\begin{eqnarray*}\mathcal{\tilde{E}}(v,\psi)=\frac{1}{2}\int_{\R_{\varepsilon}^d}\nabla |x|_{\rho}\cdot \nabla{\psi} dx+\frac{q}{2}\int_{(\R_+\cup \{0\})\times S_{\varepsilon}^{d-1}}\nabla (-x^{(1)})\cdot \nabla{\psi} dm^{(1,d-1)}\\
	=-\frac{1}{2}\int_{\R_{\varepsilon}^d}\frac{d-1}{|x|}\psi dx+\frac{1}{2}\int_{\partial \R_{\varepsilon}^d}\psi (x)\frac{\partial |x|_{\rho}}{\partial {\bf n}}\sigma(dx)-\frac{q}{2}\int_{S_{\varepsilon}^{d-1}}\int_0^{\infty} \frac{\partial \psi}{\partial x^{(1)}}dm^{(1,d-1)}\\
=-\int_M \frac{d-1}{|2x|}\psi {\bf 1}_{\R_{\varepsilon}^{d}}dx -\int_M \frac{\psi}{2} {\bf 1}_{\partial \R_{\varepsilon}^d} d\sigma +\int_M \frac{q\psi}{2} {\bf 1}_{\{0\} \times S_{\varepsilon}^{d-1}} d\sigma \ \ 
=\ \ -\int_M \psi d\nu, \end{eqnarray*}
where {\bf n} is the outward normal vector of the surface $\partial \R_{\varepsilon}^d$, $\sigma$ is the surface measure on $\partial \R_{\varepsilon}^d=\{0\}\times S_{\varepsilon}^{d-1}$, and
 $$\nu(dx):=\frac{d-1}{2|x|}{\bf 1}_{\R_{\varepsilon}^d}(x)dx+\frac{1-q}{2}d\sigma.$$
By {\cite[Theorem 5.5.5]{FOT}}, it holds that
\begin{equation} dN_{t}^{[v]}=\frac{(d-1){\bf 1}_{\{\tilde{Y}_t>0 \}}}{2(\tilde{Y}_t+\varepsilon)}dt +\frac{1-q}{2}dL_t \label{eq:hatY}\end{equation}
Here, $L$ is the positive continuous additive functional of $\tilde{X}$ whose Revuz measure is $\sigma$. By the same proof as that of Proposition \ref{dY_t}, it holds that $M^{[v]}$ is one-dimensional Brownian motion $\hat{B}$, and $\hat{L}^0_t(\tilde{Y})=\frac{1+q}{2} dL_t,$ where $\hat{L}^0_t(\tilde{Y})$ is a symmetric semimartingale local time of $\tilde{Y}$ at $0$.
Combining these with (\ref{eq:hatY}) and $q=p/|\partial \R_{\varepsilon}^d|$, $\tilde{Y}$ satisfies (\ref{eq:dY1d}). By {\cite[Theorem 2.1] {AB}}, weak solutions of (\ref{eq:dY1d}) have the same law, so this completes the proof.  \end{proof}

\begin{proof}[Proof of Theorem \ref{d'=1}]\ \\
{\bf Step1} (the case of $x$ or $y\in \R_+$)\ Fix large $T>0$ and $t\geq T$. For $f\in C_c(\R_+\cup \R_{\varepsilon}^d \cup \{a^*\})$ with ${\rm supp}(f)\subset \R_+$, we define $\tilde{f}:M\to \R$ by 

\begin{eqnarray*}\tilde{f}(\tilde{y}):= \left \{ \begin{split}&f(\tilde{y}^{(1)})&& :\tilde{y}=(\tilde{y}^{(1)},\tilde{y}^{(2)})\in \R_+\times S_{\varepsilon}^{d-1} \\
&\ \ \ 0 && :{\rm otherwise}.\end{split} \right. \end{eqnarray*}
For $x\in \R_+\cup \R_{\varepsilon}^d\cup \{a^*\}$ and $x_2\in S_{\varepsilon}^{d-1},$ we define $\tilde{x}\in M$ by 
\begin{eqnarray} \tilde{x}:=\tilde{x}(x_2):= \left \{ \begin{split}&(x,{x}_2)&&: x\in \R_+\cup \{a^*\},\\& \hspace{4mm}x &&: x\in \R_{\varepsilon}^d .\end{split} \right. \label{eq:tildex} \end{eqnarray}
Here, we defined $(a^*,x_2):=(0,x_2).$ Now, we take $x_2, x_2^*\in S_{\varepsilon}^{d-1}$ and define $\tilde{x}(x_2),\tilde{x}(x_2^*)\in M$ as in (\ref{eq:tildex}). Then, since $\tilde{f}$ is independent of $x_2$ and $x_2^*$, it holds that $\mathbb{E}_{\tilde{x}(x_2)}(\tilde{f}(\tilde{X}_t))=\mathbb{E}_{\tilde{x}(x_2^*)}(\tilde{f}(\tilde{X}_t))$, so we simply write $\tilde{x}(x_2)$ as $\tilde{x}$.
	
By Theorem \ref{proj1}, we have $\mathbb{E}_x(f(X_t))=\mathbb{E}_{u(x)}(f(-Y_t))=\mathbb{E}_{v(\tilde{x})}(f(-\tilde{Y}_t))=\mathbb{E}_{\tilde{x}}(\tilde{f}(\tilde{X}_t))$.
 While, we have
	\begin{eqnarray*}\mathbb{E}_{\tilde{x}}(\tilde{f}(\tilde{X}_t))&=&\int_{\R_+\times S_{\varepsilon}^{d-1}}\tilde{f}(\tilde{y})\tilde{p}(t,\tilde{x},\tilde{y})\tilde{m}_q(d\tilde{y})\\
	&\asymp &\int_{\R_+}f(y)\left( \int_{S_{\varepsilon}^{d-1}}\tilde{p}(t,\tilde{x},(y,y_2))dy_2\right)m_p(dy).\end{eqnarray*}
	Thus, for $x\in \R_+\cup \R_{\varepsilon}^d\cup \{a^*\}$ and $y\in \R_+$, we have
	\begin{equation} p(t,x,y)\asymp \int_{S_{\varepsilon}^{d-1}}\tilde{p}(t,\tilde{x},({y},y_2))d{y}_2. \label{eq:proj3b}\end{equation}
	
	We next consider the relation between the distance $d$ on $M$ and $\rho$ on $\R_+\cup \R_{\varepsilon}^d\cup \{a^*\}$.
	\begin{enumerate}
	\item (Figure \ref{1fig}, left) For $x,y\in \R_+$, since $S_{\varepsilon}^{d-1}$ is bounded, there exists a constant $C>0$ with $\rho (x,y)\leq d(\tilde{x},\tilde{y})\leq C+\rho(x,y)$. Hence, for $t\geq T$, it holds that $\e \asymp e^{-{d(\tilde{x},\tilde{y})^2}/{t}}$ and  $|\tilde{x}|_e \asymp 1+d(\tilde{x},K)=|x|=|x|_{\rho}$.
	
	\item (Figure \ref{1fig}, right) For $x\in \R_{\varepsilon}^d\cup \{a^*\}, y\in \R_+$, since $S_{\varepsilon}^{d-1}$ is bounded, there exists a constant $C>0$ with $\rho (x,y)\leq d(x,\tilde{y})\leq C+\rho(x,y)$. Then for $t\geq T$, it holds that $\e \asymp e^{-{d(x,\tilde{y})^2}/{t}}$.
	\end{enumerate}
Thus, for $x$ or $y\in \R_+$, the desired estimates follow from (\ref{eq:proj3b}), Theorem \ref{GSc1} and the boundedness of $S_{\varepsilon}^{d-1}.$ In particular, for all $x\in \R_+$ and $t\geq T$, it holds that $p(t,x,x)\asymp t^{-3/2}$ and by continuity of $p$, we have $p(t,a^*,a^*)\asymp t^{-3/2}$ for $t\geq T$. Combining this with the small time estimates, we have
\begin{equation}
p(t,a^*,a^*)\asymp t^{-1/2}\wedge t^{-3/2} \hspace{5mm} {\rm for}\ t>0.
\label{eq:on3aa}\end{equation}

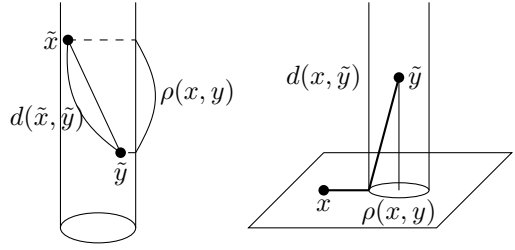
\begin{figure}[h] 
\hspace{20mm}
\begin{tikzpicture}

\draw  (2,4) ++ (-2,0) circle (0.5cm and 0.2cm);
\draw (-0.5,4)--(-0.5,7);
\draw (0.5,4)--(0.5,7);

 \draw (-0.4,6.5)node[left]{$\tilde{x}$};
 \fill (-0.4,6.5)circle (2pt);
 \draw (0.3,5)node[below]{$\tilde{y}$};
 \fill (0.3,5)circle (2pt);
 
 \coordinate (A) at (-0.4,6.5);
\coordinate (B) at (0.3,5);
 \draw [bend right,distance=0.7cm] (A)
to (B);
 \coordinate (c) at (0.5,6.5);
\coordinate (d) at (0.5,5);
\draw (c) 
(d);
\draw [bend left,distance=0.7cm] (c)
to (d);
\draw(0.3,5) -- node[below left]{$d(\tilde{x},\tilde{y})$} (-0.4,6.5);
\draw(0.7,5.8) node[right]{$\rho (x,y)$};
\draw [dashed](A) to (c);
\draw (0.4,5) to (d);

 \draw (3,5) -- (2,4) -- (4.5,4) -- (5.5,5) -- (3,5);

\draw   (4,4.5) circle (0.4cm and 0.1cm);
\draw (3.6,4.5)--(3.6,7);
\draw (4.4,4.5)--(4.4,7);
\draw (4,6)node[right]{$\tilde{y}$};
\fill (4,6)circle (2pt);
\draw (3,4.5)node[below]{$x$};
\fill (3,4.5)circle (2pt);
\draw [thick](3,4.5)--(3.6,4.5);
\draw[thick] (3.6,4.5)--(4,6);
\draw (3,6) node {$d(x,\tilde{y})$};
\draw (4,6)--(4,4.5);
\draw (4,3.9)node[above]{$\rho(x,y)$};
 
\end{tikzpicture}

\caption{the relation between $d$ and $\rho$}
\label{1fig}
\end{figure}

\noindent {\bf Step2} (the case of $x,y\in \R_{\varepsilon}^d\cup \{a^*\}$) Fix large $T\geq 2$ and $t\geq T$.\\ 
(1)For $y\in \R_{\varepsilon}^d$, by (\ref{eq:on3aa}), Lemma \ref{hitting3}, Lemma \ref{elem3} and Lemma \ref{GSh3}, we have
\begin{eqnarray}
\nonumber p(t,a^*,y)=\int_0^t p(t-s,a^*,a^*)\mathbb{P}_y(\sigma_{a^*}\in ds) \hspace{50mm} \\
\nonumber \asymp \int_0^{t/2}+\int_{t/2}^{t-1}(t-s)^{-3/2}\mathbb{P}_y(\sigma_{a^*}\in ds)+\int_{t-1}^{t}(t-s)^{-1/2}\mathbb{P}_y(\sigma_{a^*}\in ds)\hspace{4mm}\\
\nonumber \asymp \frac{\mathbb{P}_y(\sigma_{a^*}\leq t/2)}{t^{3/2}}+\!\!\left(\int_{t/2}^{t-1}\!\!(t-s)^{-3/2}ds+\int_{t-1}^{t}\!\!(t-s)^{-1/2}ds \right)\frac{|y|_{\rho}}{|y|}\frac{\ey}{t^{d/2}}\\
\asymp \frac{1}{t^{3/2}}\frac{\ey}{|y|^{d-2}}+\frac{|y|_{\rho}}{|y|}\frac{\ey}{t^{d/2}} \hspace{66mm} \label{eq:off1ayu}
\end{eqnarray}
If $1\leq |y|_{\rho}$, by $\frac{1}{1+\varepsilon}\leq \frac{|y|_{\rho}}{|y|}\leq 1$ and (\ref{eq:off1ayu}), it holds that 
$$p(t,a^*,y)\asymp \left(\frac{1}{t^{d/2}}+\frac{1}{t^{3/2}|y|^{d-2}}\right)\ey.$$
If $1> |y|_{\rho}$, by $$\frac{1}{t^{3/2}|y|^{d-2}}\ey \geq \frac{T^{(d-3)/2}}{t^{d/2}(1+\varepsilon)^{d-2}}\ey,\ \frac{|y|_{\rho}}{|y|}\leq 1$$ and (\ref{eq:off1ayu}), it holds that 
$$p(t,a^*,y)\asymp \left(\frac{1}{t^{d/2}}+\frac{1}{t^{3/2}|y|^{d-2}}\right)\ey.$$
(2)For $x,y\in \R_{\varepsilon}^d$, with $|x|_{\rho}\wedge |y|_{\rho}>1$ and $t\geq T$, by $(1)$, Lemma \ref{GSh3},  Theorem \ref{small}, Lemma \ref{hitting3}, Proposition \ref{part kernel}, Lemma \ref{rholem} and Lemma \ref{elem3}, we obtain that

\begin{eqnarray*}p(t,x,y)&=&\int_0^{t/2}+\int_{t/2}^{t-1}+\int_{t-1}^{t}p(t-s,a^*,y)\mathbb{P}_{x}(\sigma_{a^*}\in ds)+p_{\R_{\varepsilon}^d}(t,x,y)\\
&\asymp& \left(\frac{\ey}{t^{d/2}}+\frac{\ey}{t^{3/2}|y|^{d-2}}\right) \frac{\ex}{|x|^{d-2}}\\
&&\hspace{10mm}+\int_{t/2}^{t-1}+\int_{t-1}^t p(t-s,a^*,y) \frac{\ex}{t^{d/2}}ds+\frac{1}{t^{d/2}}\ee\\
&\lesssim& \frac{1}{t^{3/2}|x|^{d-2}|y|^{d-2}}\exy+\frac{1}{t^{d/2}}\e\end{eqnarray*}
and\begin{eqnarray*}p(t,x,y)&\geq&\int_0^{t/2}p(t-s,a^*,y)\mathbb{P}_{x}(\sigma_{a^*}\in ds)+p_{\R_{\varepsilon}^d}(t,x,y)\\
&\gtrsim &\frac{1}{t^{3/2}|x|^{d-2}|y|^{d-2}} \exy +\frac{1}{t^{d/2}}\e.\end{eqnarray*}

\noindent (3)For $x,y\in \R_{\varepsilon}^d$ with $|x|_{\rho}\wedge |y|_{\rho}\leq 1$ and $t\geq T$, we may assume $|x|_{\rho}\leq 1$ without loss of generality since $p(t,x,y)$ is symmetric. By $(1)$, $|x|\asymp 1$, Lemma \ref{GSh3},  Theorem \ref{small}, Lemma \ref{hitting3}, Proposition \ref{part kernel}, Lemma \ref{rholem} and Lemma \ref{elem3}, we obtain that
\begin{eqnarray*}p(t,x,y)&=&\int_0^{t/2}+\int_{t/2}^{t-1}+\int_{t-1}^{t}p(t-s,a^*,x)\mathbb{P}_{y}(\sigma_{a^*}\in ds)+p_{\R_{\varepsilon}^d}(t,x,y)\\
&\asymp& \left(\frac{\ex}{t^{d/2}}+\frac{|x|\ex}{t^{3/2}|y|^{d-2}}\right) \ey \\
&&\hspace{0mm}+\int_{t/2}^{t-1}+\int_{t-1}^t p(t-s,a^*,x) \frac{\ey}{t^{d/2}}ds+\frac{|x|_{\rho}(1\wedge |y|_{\rho})}{t^{d/2}}\ee\\
&\lesssim& \left(\frac{1}{t^{d/2}}+\frac{|x|}{t^{3/2}|y|^{d-2}}\right)\e\hspace{7mm}\end{eqnarray*}and if $|y|_{\rho}\leq 1$ or $2\leq |y|_{\rho}$, by Lemma \ref{rholem}, then we have
$$p(t,x,y)\geq \int_0^{t/2}p(t-s,a^*,x)\mathbb{P}_{y}(\sigma_{a^*}\in ds)\asymp \left(\frac{1}{t^{d/2}}+\frac{|x|}{t^{3/2}|y|^{d-2}}\right)\e.$$
Since there exists constant $c>0$ such that it holds that $\mathbb{P}_x(\sigma _{a^*}\leq 1)>c$ for $x$ with $|x|_{\rho}\leq 1$, if $1<|y|_{\rho}<2$, then we have
\begin{eqnarray*}
 p(t,x,y)&\geq & \int_0^1 p(t-s,a^*,y)\mathbb{P}_x(\sigma_{a^*}\in ds)\\
 &\gtrsim &\int_0^1 \left(\frac{1}{t^{d/2}}+\frac{|x|}{t^{3/2}|y|^{d-2}}\right)e^{-|y|_{\rho}^2/t-1}\mathbb{P}_x(\sigma_{a^*}\in ds)\\
 &\gtrsim & \left(\frac{1}{t^{d/2}}+\frac{|x|}{t^{3/2}|y|^{d-2}}\right)e^{-1/T-1} \gtrsim \left(\frac{1}{t^{d/2}}+\frac{|x|}{t^{3/2}|y|^{d-2}}\right)\e.
\end{eqnarray*}
This completes the proof of Theorem \ref{d'=1}. \end{proof}

\begin{Rem}
	\rm{In \cite{GIS}, the heat kernel estimate for Brownian motion on $(\R_+\times S_{\varepsilon}^1)\# \R^2$ is obtained. Therefore, by the same way as in this section, we can obtain the large time estimate on $\R_+\cup \R_{\varepsilon}^2\cup \{a^*\}$. By elementary computations, this estimate is the same as the one appearing in \cite{CL}.}
\end{Rem}

\section{Large time estimate($d'\geq 3$)} \label{lsec3}
In this section, we will prove Theorem \ref{d'=3}. We assume $d\geq d'\geq 3$. Moreover, we may assume $\varepsilon ,\varepsilon' <1$ without loss of generality. Unlike the case $d’=1$ in Section \ref{lsec1}, we cannot project $(\R ^{d'}\times S^{d-d'})\# \R^d$ to get $\R^{d'}_{\varepsilon'}\cup \R^d_{\varepsilon} \cup \{a^*\}$ when $d'\geq 2$. Hence, we will take careful approach. \\
For $x,y\in \R_{\varepsilon}^d$, it holds that $$p(t,x,y)=p_{\R_{\varepsilon}^d}(t,x,y)+\int _0^t \int _0 ^{t-s}p(t-s-u,a^*,a^*)\mathbb{P}_x(\sigma_{a^*} \in du)\mathbb{P}_y(\sigma_{a^*} \in ds).$$
For $x\in \R_{\varepsilon}^d, y\in \R_{\varepsilon'}^{d'}$, it holds that
$$p(t,x,y)=\int _0^t \int _0 ^{t-s}p(t-s-u,a^*,a^*)\mathbb{P}_x(\sigma_{a^*} \in du)\mathbb{P}_y(\sigma_{a^*} \in ds).$$
So, we consider the estimate of $p(t,a^*,a^*)$ in order to prove Theorem \ref{d'=3}.
 
\begin{Prop}\label{on3uu} For $t>0$, we have $$p(t,a^*,a^*)\lesssim \frac{1}{\sqrt{t}}\wedge \frac{1}{t^{{d'}/{2}}}.$$ \end{Prop}
\begin{proof}
	 For $t> 1$, it holds that $t^{-d'/2}\leq t^{-1/2}$ and $p(t,a^*,a^*)\lesssim t^{-d'/2}$ by Proposition \ref{Nash}. For $t\leq 1$, it holds that $t^{-1/2}\leq t^{-d'/2}$ and $p(t,a^*,a^*)\lesssim t^{-1/2}$ by the small time estimate (Theorem \ref{small}). Thus for $t>0$, we have $$p(t,a^*,a^*)\lesssim \frac{1}{\sqrt{t}}\wedge \frac{1}{t^{{d'}/{2}}}.$$\end{proof}
	
	\begin{Prop}\label{on3} For $t>0$, \begin{equation}p(t,a^*,a^*)\asymp \frac{1}{\sqrt{t}}\wedge \frac{1}{t^{{d'}/{2}}}.\label{eq:on3}\end{equation} \end{Prop}
	\begin{proof}
	Take $t\geq 2$ and $x\in \R_{\varepsilon'}^{d'}$ with $\sqrt{t}\leq |x| \leq 2\sqrt{t}$. For $s>0$ with $t-1<s<t$, it holds that $t-s<1$, so we can apply Theorem \ref{small} to $p(t-s,a^*,a^*)$. Thus, by Theorem \ref{small} and Lemma \ref{hitting3}, we have
\begin{eqnarray} \nonumber p(t,a^*,x)&=&\int_0^t p(t-s,a^*,a^*)\mathbb{P}(\sigma_{a^*}\in ds)\geq  \int_{t-1}^t p(t-s,a^*,a^*)\mathbb{P}(\sigma_{a^*}\in ds) \\
&\gtrsim& \int_{t-1}^t (t-s)^{-{1}/{2}} \frac{|x|_{\rho}}{|x|}\frac{e^{-{|x|_{\rho}^2}/{s}}}{s^{{d'}/{2}}+s^{{3}/{2}}|x|^{{(d'-3)}/{2}}}ds. \label{eq:on3lll}\end{eqnarray}
Since ${t}/{2}\leq t-1$ and $\sqrt{2}-\varepsilon \leq |x|_{\rho}$, we have
\begin{eqnarray} p(t,a^*,x) \gtrsim \frac{\ex}{t^{{d'}/{2}}+t^{{3}/{2}}|x|^{{(d'-3)}/{2}}}\geq \frac{e^{-{(2\sqrt{t})^2}/{t}}}{t^{{d'}/{2}}+t^{{3}/{2}}(2\sqrt{t})^{{(d'-3)}/{2}}}
\gtrsim \frac{1}{t^{{d'}/{2}}}.\label{eq:on3llll}\end{eqnarray}
By the Markov property and (\ref{eq:on3llll}), we have
\begin{eqnarray} \nonumber p(2t,a^*,a^*)&\geq&\int_{\{x\in\R_{\varepsilon'}^{d'};\sqrt{t}\leq |x|\leq 2\sqrt{t}\}} p(t,a^*,x)^2m_p(dx)\\
\nonumber&\gtrsim& \int_{\{x\in\R_{\varepsilon'}^{d'};\sqrt{t}\leq |x|\leq 2\sqrt{t}\}} t^{-d'}m_p(dx)\\
&=&\int_{\sqrt{t}}^{2\sqrt{t}}r^{d'-1}dr\times pt^{-d'} \asymp \frac{1}{t^{{d'}/{2}}},\label{eq:on3l5}\end{eqnarray}
where we used polar coordinates $r:=|x|$. (\ref{eq:on3l5}) and the small time estimate (Theorem \ref{small}) imply $p(t,a^*,a^*)\gtrsim \frac{1}{\sqrt{t}}\wedge \frac{1}{t^{{d'}/{2}}}$ for $t>0$. Thus (\ref{eq:on3}) follows from it and Proposition \ref{on3uu}.
\end{proof}

We will prove Theorem \ref{d'=3}, by using the on-diagonal estimate at $a^*$ and  hitting probability.

\begin{Prop}\label{off3u}
	Let $d\geq d'\geq3$. Then $p(t,x,y)$ satisfies the following estimates when $1\leq t:$
	\begin{enumerate}
 	\item For $x,y \in \R_{\varepsilon'}^{d'}$, $p(t,x,y)\lesssim {t^{-{d'}/{2}}}\e.$
 	\item For $x,y \in \R_{\varepsilon}^{d}$ with $|x|_{\rho}\vee |y|_{\rho} \leq 1$, $p(t,x,y)\lesssim {t^{-{d'}/{2}}}\e.$\\
 	For $x,y \in \R_{\varepsilon}^{d}$ with $|x|_{\rho}\vee |y|_{\rho} > 1$, $$p(t,x,y)\lesssim \frac{1}{t^{{d'}/{2}}|x|^{d-2}|y|^{d-2}}\exy+\frac{1}{t^{{d}/{2}}}\e.$$
 	\item For $x\in \R_{\varepsilon}^{d}\cup \{a^*\}, y\in \R_{\varepsilon'}^{d'}\cup \{a^*\}$, $$p(t,x,y)\lesssim \left( \frac{1}{t^{{d}/{2}}|y|^{{d'}-2}}+\frac{1}{t^{{d'}/{2}}|x|^{d-2}}\right) \e.$$
	\end{enumerate}
\end{Prop}

\begin{proof}
	In order to avoid a long calculation, we will prove the estimates by comparing $\R_{\varepsilon}^d \cup \R_{\varepsilon'}^{d'}\cup \{a^*\}$ with a manifold with ends. First, we assume $\varepsilon \leq \varepsilon'$ (See Figure \ref{off3ufig}).
	
	\begin{figure}[h]
\hspace{2mm}
\begin{tikzpicture}
\draw (-3,4)--(-4,3)--(0,3)--(1,4)--(-3,4);
\draw[dashed] (-3,4)--(-3,1);
\draw (-4,3)--(-4,0);
\draw (0,3)--(0,0);
\draw (1,4)--(1,1);
\draw (-3,1)--(-4,0)--(0,0)--(1,1)--(-3,1);

\draw (-1.6,3.5) circle (0.45 and 0.2);
\draw (-1.6,0.5) circle (0.45 and 0.2);

\draw (-2.05,3.5)--(-2.05,0.5);
\draw (-1.15,3.5)--(-1.15,0.5);

\coordinate (A) at (1,4);
\coordinate (B) at (1,1);
\draw (A) 
(B);
\draw [bend left,distance=1.3cm] (A)
to node [fill=white,inner sep=0.01pt,circle] {$S_{\varepsilon}^{d-d'}$} (B);
\draw (0.6,3.8) node{$\R^{d'}$};
\draw (-1.6,3.5)--(-1.15,3.5)node[above]{$\varepsilon'$};

\node[circle,shading=ball, outer color=red!70, inner color=pink, minimum width=6mm][label=below:$S_{\varepsilon}^d$] (ball) at (-1.6,2) {};

\end{tikzpicture}
\tdplotsetmaincoords{65}{45}
\begin{tikzpicture}[tdplot_main_coords]
 \draw[->,>=stealth,very thick] (-3,0,0)--(3,0,0)node[above]{};
 \draw[->,>=stealth,very thick] (0,-2.5,0)--(0,3,0)node[right]{};
 \draw[->,>=stealth,very thick] (0,0,-2)--(0,0,3)node[right]{};
 \draw (1.2,1.0,2.2)--(1.2,1.0,2.8);
 \draw (1.2,1.0,2.2)--(1.6,1.4,2.2);
 \draw (1.5,1.2,2.5)node{$\R^{d}_{\varepsilon}$};
\node[circle,shading=ball, outer color=red!70, inner color=pink, minimum width=6mm][label=left:$S_{\varepsilon}^d$] (ball) at (0,0) {};
\end{tikzpicture}
 
\caption{$(\R^{d'}\times S_{\varepsilon}^{d-d'})\# \R^d$}
\label{off3ufig}
 \end{figure}
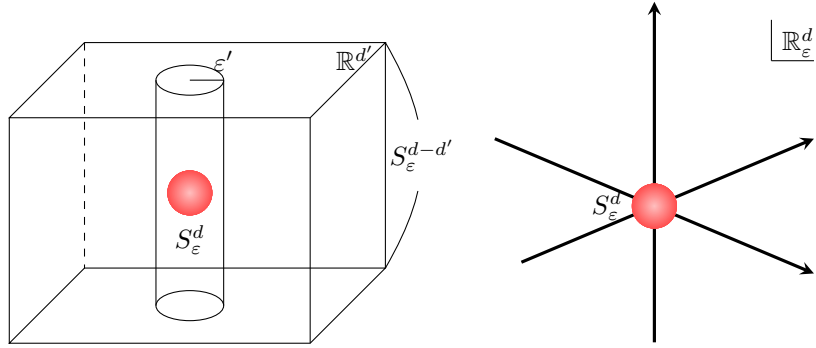
	Let $\tilde{p}(t,x,y)$ be the heat kernel of Brownian motion $\tilde{X}$ on $(\R^{d'}\times S_{\varepsilon}^{d-d'})\# \R^d$, where $S_{\varepsilon}^{d-d'}:=\{x\in \R^{d-d'+1}: |x|=\varepsilon \}$. According to {\cite[Example 4.5 and Example 5.5]{GSc}}, for $t>1$, $\tilde{p}(t,x,y)$ has sharp estimates as the right hands side of  this proposition up to the difference between distances $\rho$ and $d$, where $d$ is a geodesic distance on  $(\R^{d'}\times S_{\varepsilon}^{d-d'})\# \R^d$. Furthermore, let $K:=(\overline{B}^{d'}(0;\varepsilon') \times S_{\varepsilon}^{d-d'}) \cup \overline{B}^{d}(0;\varepsilon)$ then, for $t>1$ and $\tilde{x},\tilde{y}\in K$, it holds that $\tilde{p}(t,\tilde{x},\tilde{y})\asymp t^{-{d'}/{2}}$. Here $\overline{B}^d$ is a closed ball on $\R^d$. By combining this with small time estimates ({\cite[Theorem 5.10]{GSc}}), for $t>0$ and $\tilde{x}, \tilde{y}\in K$, we obtain $\tilde{p}(t,\tilde{x},\tilde{y})\asymp t^{-{d}/{2}}e^{-d(\tilde{x},\tilde{y})/t} \vee t^{-{d'}/{2}}$.

	By proposition \ref{on3}, we have $p(t,a^*,a^*)\asymp t^{-{1}/{2}} \wedge t^{-{d'}/{2}}\leq t^{-{d}/{2}} \vee t^{-{d'}/{2}}$,
	 \begin{eqnarray*}&&\mathbb{P}_{x}(\sigma_{a^*}\in ds)=\tilde{\mathbb{P}}_{x}(\tilde{\sigma}_K\in ds),\ p_{\R_{\varepsilon}^d}(t,x,y)=\tilde{p}_{\R^d\setminus K}(t,x,y)\  {\rm for} \ x,y\in \R^d_{\varepsilon}\  {\rm and}\\
	&& \mathbb{P}_{x}(\sigma_{a^*}\in ds)=\tilde{\mathbb{P}}_{(x,x_2)}(\tilde{\sigma}_K\in ds)\  {\rm for\ } x\in \R^{d'}_{\varepsilon'}, \ x_2\in S_{\varepsilon}^{d-d'}, \end{eqnarray*}
	where $\tilde{\mathbb{P}}$, $\tilde{\sigma}_K$ and $\tilde{p}_{\R^d\setminus K}$ are those for the process $\tilde{X}$. Moreover, for the part process on $(\R^{d'}\times S_{\varepsilon}^{d-d'})\setminus K =\mathbb{R}^{d'}_{\varepsilon'}\times S_{\varepsilon}^{d-d'}$ of $\tilde{X}$,  the projection to $\R^{d'}_{\varepsilon'}$ has the same law as the part process on $\R^{d'}_{\varepsilon'}$ of $X$. Thus, for $x,y\in \R_{\varepsilon'}^{d'},\ x_2\in S_{\varepsilon}^{d-d'}$, by the same reason as the proof of Theorem \ref{d'=1}, and continuity of $\tilde{p}$,
	\begin{eqnarray*}p_{\R_{\varepsilon'}^{d'}}(t,x,y)&=&\int_{S_{\varepsilon}^{d-d'}} \tilde{p}_{(\R^{d'}\times S_{\varepsilon}^{d-d'}) \setminus K}(t,(x,x_2),(y,y_2))dy_2\\
	&\lesssim & \max_{y_0\in S_{\varepsilon}^{d-d'}}\tilde{p}_{(\R^{d'}\times S_{\varepsilon}^{d-d'}) \setminus K}(t,(x,x_2),(y,y_0)).\end{eqnarray*}
	Hence we have
	\begin{eqnarray*}p(t,x,y)=p_{\R_{\varepsilon'}^{d'}}(t,x,y)+\int _0^t\int_0^{t-s} p(t-s-w,a^*,a^*)\mathbb{P}_{x}(\sigma_{a^*}\in dw)\mathbb{P}_{y}(\sigma_{a^*}\in ds)\\
	=p_{\R_{\varepsilon'}^{d'}}(t,x,y)+\tilde{\mathbb{E}}_{\tilde{x}}\tilde{\mathbb{E}}_{\tilde{y}}\int _0^t\int_0^{t-s} p(t-s-w,a^*,a^*)\mathbb{P}_{x}(\sigma_{a^*}\in dw)\mathbb{P}_{y}(\sigma_{a^*}\in ds)\\
	\lesssim \tilde{p}_{(\R^{d'}\times S_{\varepsilon}^{d-d'}) \setminus K}(t,\tilde{x},\tilde{y}) \hspace{80mm}\\
	 +\tilde{\mathbb{E}}_{\tilde{x}}\tilde{\mathbb{E}}_{\tilde{y}}\int _0^t\! \int_0^{t-s} \hspace{-3mm}\tilde{p}(t-s-w,\tilde{X}_s,\tilde{X}_w)\tilde{\mathbb{P}}_{\tilde{x}}(\tilde{\sigma}_K\in dw)\tilde{\mathbb{P}}_{\tilde{y}}(\tilde{\sigma}_K\in ds)\\
	=\tilde{p}(t,\tilde{x},\tilde{y}), \hspace{100mm}  \end{eqnarray*}
	where we denote $\tilde{x}:=(x,x_2),\ \tilde{y}:=(y,y_2)$ for $x,y\in \R_{\varepsilon'}^{d'}$ and $x_2,y_2\in S_{\varepsilon}^{d-d'}$
	with $$\max_{y_0\in S_{\varepsilon}^{d-d'}}\tilde{p}_{(\R^{d'}\times S_{\varepsilon}^{d-d'}) \setminus K}(t,(x,x_2),(y,y_0))=\tilde{p}_{(\R^{d'}\times S_{\varepsilon}^{d-d'}) \setminus K}(t,(x,x_2),(y,y_2)).$$
	In the above inequalities, we used the following estimates in order to treat the effect of $e^{-d(\tilde{x},\tilde{y})/t}$ appearing in the estimate of $\tilde{p}(t,\tilde{x},\tilde{y})$ for $t<1, \tilde{x},\tilde{y}\in K$. For $x,y\in \R^d_{\varepsilon}$, we have
	\begin{eqnarray*} \int_{\{0\leq t-s-w\leq1,\ s\geq w\}}p(t-s-w,a^*,a^*)\mathbb{P}_y(\sigma_{a^*}\in ds)\mathbb{P}_x(\sigma_{a^*}\in dw)\hspace{15mm}\\
	\lesssim \int_0^1\hspace{-2mm}\int_{t-1}^{t-w}\hspace{-3mm} +\!\! \int_0^{t/2}\hspace{-3mm}\int _{(t-w-1)\vee (t/2-1)}^{t-1}\hspace{-4mm}+\int_{(t-1)/2}^{t/2} \int_{w\vee(t-w-1)}^{t/2}\hspace{-14mm}(t-s-w)^{-1/2}ds \frac{\ey}{t^{d/2}}\mathbb{P}_x(\sigma_{a^*}\in dw)\\
	\leq 2\frac{\ey}{t^{d/2}}\mathbb{P}_x(\sigma_{a^*}\leq t)+\int_{(t-1)/2}^{t/2}\left(\frac{t}{2}-w\right)^{-1/2}\frac{\ey}{t^{d/2}}\mathbb{P}_x(\sigma_{a^*}\in dw)\hspace{14mm}\\
	\lesssim \frac{\exy}{|x|^{d-2}t^{d/2}}+\frac{\exy}{t^{d}}\ \lesssim \tilde{p}(t,x,y).\hspace{48mm}\end{eqnarray*}
	Thus, by the symmetry, we have
	$$\int_{\{t-s-w\leq1 \}}p(t-s-w,a^*,a^*)\mathbb{P}_x(\sigma_{a^*}\in dw)\mathbb{P}_y(\sigma_{a^*}\in ds) \lesssim \tilde{p}(t,x,y).$$
	The same inequalities hold for the cases of $x\in \R_{\varepsilon}^d, y\in \R_{\varepsilon'}^{d'}$ and $x,y\in \R_{\varepsilon'}^{d'}$.

	By the compactness of $K$, we can ignore the difference between $\rho$ and $d$ and derive upper estimates similarly as in the proof of Theorem \ref{d'=1}.

	If $\varepsilon >\varepsilon'$, we can prove in the same way as above by exchanging $(\R^{d'}\times S_{\varepsilon}^{d-d'})\# \R^d$ and $K$ to $(\R^{d'}\times S_{\varepsilon'}^{d-d'})\# \R^d$ and $(\overline{B}^{d'}(0;\varepsilon') \times S_{\varepsilon'}^{d-d'}) \cup \overline{B}^{d}(0;\varepsilon)$, respectively.
\end{proof}

\begin{Rem}
{\rm One can prove Proposition \ref{off3u} directly by using the estimates of $p(t,a^*,a^*)$ and $\mathbb{P}_x(\sigma_{a^*}\in ds)$.}
\end{Rem}

\begin{proof}[Proof of Theorem \ref{d'=3}]
The upper estimates is already proved in Proposition \ref{off3u}, so we consider the lower estimates. In this proof, let $T>3$ be large, and $t\in [T,\infty)$.\\
{\bf Step1} (the estimate of $p(t,x,a^*)$) \\
$($1$)$ For $x\in \R_{\varepsilon}^d$ with $|x|_{\rho}\geq1$, by the Markov property, Theorem \ref{small}, (\ref{eq:on3}), Lemma \ref{hitting3}, Lemma \ref{elem3} and Lemma \ref{GSh3}, we have
 \begin{eqnarray*}p(t,x,a^*)\geq \int_0^{{t}/{2}}p(t-s,a^*,a^*)\mathbb{P}_x(\sigma _{a^*}\in ds)+\int_{t-1}^tp(t-s,a^*,a^*)\mathbb{P}(\sigma _{a^*}\in ds)\\
 \asymp t^{-{d'}/{2}}\mathbb{P} \left(\sigma_{a^*}\leq \frac{t}{2}\right)+\int_{t-1}^t (t-s)^{-{1}/{2}}\frac{|x|_{\rho}}{|x|}\frac{\ex}{t^{{d}/{2}}+t^{{3}/{2}}|x|^{{(d-3)}/{2}}}ds \hspace{3mm}\\
 \gtrsim \left( \frac{1}{t^{{d'}/{2}}|x|^{d-2}}+\frac{1}{t^{{d}/{2}}} \right) \ex. \hspace{55mm}
 \end{eqnarray*}
 For $x\in \R_{\varepsilon}^d$ with $|x|_{\rho}<1$, by the Markov property, (\ref{eq:on3}) and Lemma \ref{GSh3}, we have
 \begin{eqnarray*}p(t,x,a^*)\geq \int_0^{{t}/{2}}p(t-s,a^*,a^*)\mathbb{P}_x(\sigma _{a^*}\in ds) \gtrsim \frac{1}{t^{{d'}/{2}}}\ex.\end{eqnarray*}
 $($2$)$ For  $x\in \R_{\varepsilon'}^{d'}$, we can prove in the same way as in the case of $x\in \R_{\varepsilon}^{d}$. Since the estimate of $p(t,a^*,a^*)$ depends only on $d'$, we can derive $p(t,x,a^*)\gtrsim t^{-{d'}/{2}}\ex$ from ($1$) by changing $d$ to $d'$.\\
{\bf Step2} (Theorem \ref{d'=3} $($i$)$ and $($ii$)$) \\
 $($1$)$ For $x,y\in \R_{\varepsilon}^d$, by (\ref{eq:on3}), (\ref{eq:part kernel}), Step1, Lemma \ref{rholem} and Lemma \ref{GSh3}, we have
  \begin{eqnarray} p(t,x,y)&=&p_{\R_{\varepsilon}^d}(t,x,y)+\overline{p}_{\R_{\varepsilon}^d}(t,x,y) \hspace{60mm} \label{eq:off3b}\\
 \nonumber &\geq& p_{\R_{\varepsilon}^d}(t,x,y)+ \int_{0}^{{t}/{2}}p(t-s,a^*,x)\mathbb{P}_y(\sigma_{a^*}\in ds) \hspace{28mm} \\
  \nonumber &\gtrsim& \frac{\left(1\wedge |x|_{\rho}\right)\left(1\wedge |y|_{\rho}\right)}{t^{{d}/{2}}}\ee+p(t,a^*,x)\mathbb{P}_y\left(\sigma_{a^*}\leq \frac{t}{2}\right)\\
   \nonumber &\gtrsim& \frac{\left(1\wedge |x|_{\rho}\right)\left(1\wedge |y|_{\rho}\right)}{t^{{d}/{2}}}\e+ p(t,a^*,x)\frac{\ey}{|y|^{d-2}}.\end{eqnarray}
$($a$)$ If $|x|_{\rho}\vee |y|_{\rho} \leq 1$,  by (\ref{eq:off3b}) and Lemma \ref{rholem}, we have
\begin{equation*}p(t,x,y)\gtrsim 0+ 0+\frac{1}{t^{{d'}/{2}}}\e \gtrsim \frac{1}{t^{{d'}/{2}}}\e. \end{equation*}
$($b$)$ If $|x|_{\rho}>1\geq |y|_{\rho}>\frac{1}{2}$, by (\ref{eq:off3b}), we have
$$p(t,x,y)\gtrsim \frac{1}{t^{{d}/{2}}}\e+ \frac{\exy}{t^{{d'}/{2}}|x|^{d-2}|y|^{d-2}}+0.$$
$($c$)$ If $|x|_{\rho}>1,\ \frac{1}{2}\geq |y|_{\rho}$, by (\ref{eq:off3b}) and Lemma \ref{rholem} $($iii$)$,  we have
\begin{eqnarray*}p(t,x,y)\gtrsim \! \frac{\left(1\wedge |x|_{\rho}\right)\! \left(1\wedge |y|_{\rho}\right)}{t^{{d}/{2}}}\e+ \left( \frac{1}{t^{{d'}/{2}}|x|^{d-2}}\! +\! \frac{1}{t^{{d}/{2}}} \right)\! \frac{\exy}{|y|^{d-2}}\\
\gtrsim 0+ \left(\frac{1}{t^{{d'}/{2}}|x|^{d-2}}+ \frac{1}{t^{{d}/{2}}} \right)\frac{\e}{|y|^{d-2}} \hspace{47mm}\\
\asymp \frac{1}{t^{{d'}/{2}}|x|^{d-2}|y|^{d-2}}\exy+\frac{1}{t^{{d}/{2}}}\e. \hspace{28mm}
\end{eqnarray*}
By the  above estimates $($a$)$-$($c$)$ and using the symmetry of $p(t,x,y)$, we obtain the estimates in Theorem \ref{d'=3} $($ii$)$.\\
$($2$)$ For $x,y\in \R_{\varepsilon'}^{d'}$, we can prove in the same way as in the case of $x,y\in \R_{\varepsilon}^{d}$. Since the estimate of $p(t,a^*,x)$ depends only on $d'$, and we can derive $$p(t,x,y)\gtrsim \frac{\e}{t^{{d'}/{2}}}$$ from ($1$) by changing $d$ to $d'$.\\
{\bf Step3} (Theorem \ref{d'=3}$($iii$)$)\\
For $x\in \R_{\varepsilon}^d,y\in \R_{\varepsilon'}^{d'}$, by Step1, Lemma \ref{hitting3}, Lemma \ref{elem3} and Lemma \ref{GSh3}, we obtain
\begin{eqnarray}\nonumber p(t,x,y)\geq \int_0^{{t}/{2}}+\int_{{t}/{2}}^{t-1}p(t-s,a^*,y)\mathbb{P}_x(\sigma_{a^*}\in ds)\hspace{37mm}\\
\nonumber \gtrsim \! \frac{\ex \mathbb{P}_x \left( \sigma_{a^*}\leq \frac{t}{2}\right)}{t^{{d'}/{2}}}\! +\! \int_{{t}/{2}}^{t-1}\frac{e^{-{|y|^2_{\rho}}/{(t-s)}}ds}{(t-s)^{{d'}/{2}}}\frac{|x|_{\rho}}{|x|} \frac{\ex}{t^{{d}/{2}}+t^{{3}/{2}}|x|^{{(d-3)}/{2}}}\hspace{-7mm} \\
\gtrsim \frac{1}{t^{{d'}/{2}}|x|^{d-2}}\e+ \int_{{t}/{2}}^{t-1}\frac{e^{-{|y|^2_{\rho}}/{(t-s)}}}{(t-s)^{{d'}/{2}}}ds\frac{|x|_{\rho}}{|x|} \frac{\ex}{t^{{d}/{2}}}. \hspace{8mm} \label{eq:off3c}\end{eqnarray}
$($a$)$ If $|x|_{\rho}<1$, by (\ref{eq:off3c}) and $|y|\geq \varepsilon ',$ we have
$$p(t,x,y)\gtrsim \frac{1}{t^{{d'}/{2}}|x|^{d-2}}\e+0 \gtrsim \left(\frac{1}{t^{{d'}/{2}}|x|^{d-2}}+\frac{1}{t^{{d}/{2}}|y|^{d-2}}\right)\e.$$
$($b$)$ If $|x|_{\rho}\geq1,\ |y|_{\rho}\leq 1$, by (\ref{eq:off3c}) and $3<T\leq t$, we have
\begin{eqnarray*}p(t,x,y) &\gtrsim&\frac{1}{t^{{d'}/{2}}|x|^{d-2}}\e+ \int_{{t}/{2}}^{t-1}\frac{e^{-1}}{(t-s)^{{d'}/{2}}}ds\frac{\ex}{t^{{d}/{2}}}\\
&\gtrsim&\frac{1}{t^{{d'}/{2}}|x|^{d-2}}\e+ \left(1-\left(\frac{t}{2} \right)^{1-{d'}/{2}}\right) \frac{\ex}{t^{{d}/{2}}}\\
&\gtrsim&\left(\frac{1}{t^{{d'}/{2}}|x|^{d-2}}+\frac{1}{t^{{d}/{2}}|y|^{d'-2}} \right) \e.
\end{eqnarray*}
$($c$)$ If $|x|_{\rho}\geq1,\ 1<|y|_{\rho}<|y|<{\sqrt{t}}/{2}$, by (\ref{eq:off3c}) and let $\theta:=\frac{|y|^2_{\rho}}{t-s}$, we have
 \begin{eqnarray*}p(t,x,y) &\gtrsim&\frac{1}{t^{{d'}/{2}}|x|^{d-2}}\e+ \int_{{t}/{2}}^{t-1}\frac{e^{-{|y|^2_{\rho}}/{(t-s)}}}{(t-s)^{{d'}/{2}}}ds\frac{|x|_{\rho}}{|x|} \frac{\ex}{t^{{d}/{2}}}\\
& \asymp& \frac{1}{t^{{d'}/{2}}|x|^{d-2}}\e+ \int_{{2|y|_{\rho}^2}/{t}}^{|y|_{\rho}^2}e^{-\theta}\theta^{{d'}/{2}-2}d\theta \frac{\ex}{t^{{d}/{2}}|y|^{d'-2}}\\
&\gtrsim&\frac{1}{t^{{d'}/{2}}|x|^{d-2}}\e+ \int_{{1}/{2}}^{1}e^{-\theta}\theta^{{d'}/{2}-2}d\theta \frac{\ex}{t^{{d}/{2}}|y|^{d'-2}}\\
&\asymp&\left(\frac{1}{t^{{d'}/{2}}|x|^{d-2}}+\frac{1}{t^{{d}/{2}}|y|^{d'-2}} \right) \e.
\end{eqnarray*}
$($d$)$ If $|x|_{\rho}\geq1,\ {\sqrt{t}}/{2}\leq|y|$, by (\ref{eq:off3c}) and ${2t}/{3}<t-1$, we have
\begin{eqnarray*}p(t,x,y) &\gtrsim&\frac{1}{t^{{d'}/{2}}|x|^{d-2}}\e+ \int_{{t}/{2}}^{{2t}/{3}}\frac{e^{-{|y|^2_{\rho}}/{(t-s)}}}{(t-s)^{{d'}/{2}}}ds\frac{1}{t^{{d}/{2}}}\ex\\
&\asymp&\frac{1}{t^{{d'}/{2}}|x|^{d-2}}\e+ \int_{{t}/{2}}^{{2t}/{3}}t^{-{d'}/{2}}ds \frac{1}{t^{{d}/{2}}}\e\\
&\asymp&\frac{1}{t^{{d'}/{2}}|x|^{d-2}}\e+ \frac{1} {t^{{(d+d'-2)}/{2}}}\e\\
&\gtrsim&\left(\frac{1}{t^{{d'}/{2}}|x|^{d-2}}+ \frac{1} {t^{{d}/{2}}|y|^{d'-2}}\right)\e.
\end{eqnarray*}
By $($a$)$-$($d$)$, we obtain the assertion of Theorem \ref{d'=3} $($iii$)$.
\end{proof}

\section{Large time estimate($d'=2$)} \label{lsec2}
In this section, we will prove Theorem \ref{d'=2,2} and Theorem \ref{d'=2,3}. Let $d'=2,\ d\geq 2$ and without loss of generality, we assume $\varepsilon ,\varepsilon' <1$.

For a same reason as in the case of $d'=3$, we consider the estimate of $p(t,a^*,a^*)$. When $d=d'=2$, this is easy. When $d\geq 3, d'=2$, we will obtain the estimate by using Doob's $h$-transform and the relative Faber-Krahn inequality.

\begin{Prop}\label{on2weak}
	Let $d\geq d'=2$. Then, for $t>0$, it holds that
	\begin{equation}t^{-{1}/{2}}\wedge t^{-{d}/{2}} \lesssim p(t,a^*,a^*) \lesssim t^{-{1}/{2}}\wedge t^{-1}. \end{equation}
\end{Prop}
\begin{proof}
The upper estimate follows from Proposition \ref{Nash} and Theorem \ref{small}. By the Markov property and the Cauchy-Shwarz inequality, for large $M>0$, we have
	\begin{eqnarray}\nonumber p\left(t,a^*,a^*\right)&=&\int p\left(\frac{t}{2},a^*,x\right)^2m_p(dx)\geq \int_{\{|x|\leq M\sqrt{t} \}} p\left(\frac{t}{2},a^*,x\right)^2m_p(dx)\\
	\nonumber&\geq&m_p\left( \{|x|\leq M\sqrt{t} \} \right)^{-1}\times \left(\int_{\{|x|\leq M\sqrt{t} \}} p\left(\frac{t}{2},a^*,x\right)m_p(dx)\right)^2\\
	&\gtrsim& \left(\frac{1}{t}\wedge \frac{1}{t^{{d}/{2}}} \right)\mathbb{P}_{a^*}\left(|X_t|\leq M\sqrt{t} \right)^2. \label{eq:on3l}\end{eqnarray}
By the proof of {\cite[Theorem 5.10]{CL}}, there is large $M>0$ such that for all $t>0$, $\mathbb{P}_{a^*}\left(|X_t|\leq M\sqrt{t} \right)\geq \frac{1}{2}$. Thus the right hand side of (\ref{eq:on3l}) is equal to $t^{-1}\wedge t^{-{d}/{2}}$ up to a constant multiple. Therefore, by Theorem \ref{small} again, we have
\begin{equation}p(t,a^*,a^*)\gtrsim \frac{1}{\sqrt{t}}\wedge \left(\frac{1}{t} \wedge \frac{1}{t^{{d}/{2}}}\right)=\frac{1}{\sqrt{t}}\wedge \frac{1}{t^{{d}/{2}}}. \label{eq:on3ll}\end{equation}
\end{proof}

\begin{Cor}\label{on22}
	Let $d=d'=2$. Then, for $t>0$, $p(t,a^*,a^*)\asymp t^{-{1}/{2}}\wedge t^{-1}.$
\end{Cor}

\begin{proof}[Proof of Theorem \ref{d'=2,2}]
	Let $d=d'=2$. We may assume $\varepsilon \geq \varepsilon'$ without loss of generality. $\tilde{p}(t,x,y)$ denotes the heat kernel for Brownian motion $\tilde{X}$ on $\R^2 \# \R^2$. Then, by {\cite[Example 2.12]{GIS}}, $\tilde{p}(t,x,y)$ has the estimates of this theorem as a sharp estimate. In particular, it holds that $$\tilde{p}(t,x,y)\asymp t^{-1}e^{-d(x,y)/t} \ {\rm for}\  t>0\ {\rm and}\  x,y\in K:=\overline{B}^2(0;{\varepsilon}) \cup \overline{B}^2(0;{\varepsilon'}),$$ where $d$ is a geodesic distance on $\R^2\# \R^2$. By Corollary \ref{on22}, we have $p(t,a^*,a^*)\lesssim t^{-1}$. Furthermore, it holds that $\mathbb{P}_{x}(\sigma_{a^*}\in ds)=\mathbb{\tilde{P}}_{x}(\tilde{\sigma}_{K}\in ds)$ and heat kernels of part processes of $X$ and $\tilde{X}$ on $\R_{\varepsilon}^2$ are equivalent, where $\mathbb{\tilde{P}}$ and $\tilde{\sigma}_K$ are those for $\tilde{X}$. Thus, by the same way as the proof of Proposition \ref{off3u}, it holds that $p(t,x,y)\lesssim \tilde{p}(t,x,y)$ for $x,y \in \R_{\varepsilon}^2 \cup \R_{\varepsilon'}^2\cup \{a^*\}$, so the upper estimates are proved.

	Next, we prove the lower estimates. Let $T>0$ be large and $t\in [T, \infty)$.\\
	$($1$)$ $($a$)$ For $x\in \R_{\varepsilon}^2$ with $|x|\leq 1$, by Corollary \ref{on22}, we have
	\begin{equation*}p(t,x,a^*)\geq \int_0 ^{{t}/{2}}p(t-s,a^*,a^*)\mathbb{P}_{x}(\sigma_{a^*}\in ds)\gtrsim \frac{1}{t}\mathbb{P}_{x}(\sigma_{a^*}\leq 1)\asymp \frac{\ex}{t}.\end{equation*}
	$($b$)$ For $x\in \R_{\varepsilon}^2$ with $1<|x|\leq {\sqrt{t}}/{2}$, by Lemma \ref{hitting2}, Lemma \ref{GSh2}, and Corollary \ref{on22},  we have
\begin{eqnarray*}	p(t,x,a^*) \geq \int_{0}^{{t}/{2}}p(t-s,a^*,a^*)\mathbb{P}_{x}(\sigma_{a^*}\in ds)\!+\! \int_{{t}/{2}}^{t-1}p(t-s,a^*,a^*)\mathbb{P}_{x}(\sigma_{a^*}\in ds)\\
	\asymp \frac{1}{t}\left( 1-\frac{\log{|x|}}{\log{\sqrt{{t}/{2}}}} \right)+\int_{{t}/{2}}^{t-1}\frac{1}{t-s}ds\ \mathbb{P}_{x}(\sigma_{a^*}\in dt) \hspace{29mm}\\
 \gtrsim \frac{1}{t}\left( 1-\frac{\log{|x|}}{\log{\sqrt{{t}/{2}}}} \right)\ex+\frac{\log{|x|}}{t\log{t}}\ex \hspace{37mm}\\
\gtrsim \frac{1}{t}\ex. \hspace{88mm} \end{eqnarray*}
	$($c$)$ For $x\in \R_{\varepsilon}^2$ with ${\sqrt{t}}/{2}<|x|$, by Lemma \ref{hitting2}, and Corollary \ref{on22},   we have
	\begin{eqnarray*} p(t,x,a^*)\geq \int_{{t}/{2}}^{t-1}p(t-s,a^*,a^*)\mathbb{P}_{x}(\sigma_{a^*}\in ds)\asymp \int_{{t}/{2}}^{t-1}(t-s)^{-1}ds\mathbb{P}_{x}(\sigma_{a^*}\in dt)\\
	\gtrsim \log{t} \frac{1+\log{|x|}}{\left(1+\log{(1+t)} \right)\left(1+\log{(|x|^2+|x|)} \right)}\frac{\ex}{t}\asymp \frac{1}{t}\ex.\hspace{6.5mm}
	\end{eqnarray*}
$($2$)$ For $x,y\in \R_{\varepsilon}^2$, by (\ref{eq:part kernel}), Lemma \ref{rholem} and Proposition \ref{on22}, it holds that
 \begin{eqnarray}\nonumber p(t,x,y)&\geq& p_{\R_{\varepsilon}^2}(t,x,y)+\int_{0}^{{t}/{2}}p(t-s,a^*,y)\mathbb{P}_x(\sigma_{a^*}\in ds)\\
 &\asymp& \frac{(1\wedge |x|_{\rho})(1\wedge |y|_{\rho})}{t}\e+\frac{1}{t}\mathbb{P}_x (\sigma_{a^*}\leq t/2). \label{eq:off22a}\end{eqnarray}
$($a$)$ If $\frac{1}{2}\leq |x|_{\rho},\ 1\leq |y|_{\rho}$, by (\ref{eq:off22a}), we have
$$p(t,x,y)\gtrsim \frac{1}{t}\e+0 =\frac{1}{t}\e.$$
$($b$)$ If $|x|_{\rho}<\frac{1}{2},\ 1\leq |y|_{\rho}$, by (\ref{eq:off22a}) and Lemma \ref{rholem} (iii), we have
$$p(t,x,y)\gtrsim 0+\frac{1}{t}\mathbb{P}_x(\sigma_{a^*}\leq 1) \asymp \frac{1}{t}\exy\asymp \frac{1}{t}\e.$$
$($c$)$ If $|x|_{\rho}\leq 1,\ |y|_{\rho}\leq 1$, by (\ref{eq:off22a}) and Lemma \ref{rholem} (ii), we have
$$p(t,x,y)\gtrsim 0+\frac{1}{t}\mathbb{P}_x(\sigma_{a^*}\leq 1) \asymp \frac{1}{t}\exy\asymp \frac{1}{t}\e.$$
$($3$)$ For $x\in \R_{\varepsilon}^2, y\in \R_{\varepsilon'}^2$, it holds that \begin{equation}p(t,x,y)\geq \int_{0}^{{t}/{2}}p(t-s,a^*,x)\mathbb{P}_y(\sigma_{a^*}\in ds). \label{eq:off22b}\end{equation}
$($a$)$ If $|x|\vee |y|\leq \sqrt{t}/{2}$, by (\ref{eq:off22b}) and Lemma \ref{GSh2}, we have
$$p(t,x,y) \gtrsim \frac{1}{t}\left(1-\frac{\log{|y|}}{\log{\sqrt{t/2}}} \right)\asymp \frac{\frac{1}{2}\log{\frac{t}{2}}-\log{|y|}}{t\log{t}}.$$
By the symmetry, we have 
\begin{eqnarray*}p(t,x,y)&\gtrsim& \frac{\log{\frac{t}{2}}-\log{|x|}-\log{|y|}}{t\log{t}}\\
&\gtrsim&\frac{U_t(x)}{t}\left(U_t(y)+\frac{\log{|y|}}{\log{(t|y|)}} \right)+\frac{U_t(y)}{t}\left(U_t(x)+\frac{\log{|x|}}{\log{(t|x|)}} \right)\\
&\asymp&\frac{\e}{t}\left( U_t(x)U_t(y)+\frac{U_t(x)\log{|y|}}{\log{(1+t|y|)}}+\frac{U_t(y)\log{|x|}}{\log{(1+t|x|)}} \right). \end{eqnarray*}
$($b$)$ If $|x|\wedge |y|\geq {\sqrt{t}}/{2}$, by (\ref{eq:off22b}) and Lemma \ref{GSh2}, we have
$$p(t,x,y) \gtrsim \frac{1}{t}\left(1-\frac{\log{|y|}}{\log{\sqrt{t/2}}} \right)\asymp \frac{1}{t\log{|y|}}\e.$$
By the symmetry, we have
\begin{eqnarray*}p(t,x,y)&\gtrsim&\frac{1}{t}\left(\frac{1}{\log{|x|}}+ \frac{1}{\log{|y|}}\right)\e \\
&\asymp&\frac{\e}{t}\left( U_t(x)U_t(y)+\frac{U_t(x)\log{|y|}}{\log{(1+t|y|)}}+\frac{U_t(y)\log{|x|}}{\log{(1+t|x|)}} \right).\end{eqnarray*}
$($c$)$ If $|y|\leq {\sqrt{t}}/{2}\leq |x|$, by (\ref{eq:off22b}) and Lemma \ref{GSh2}, we have
$$p(t,x,y) \gtrsim \frac{1}{t}\left(1-\frac{\log{|y|}}{\log{\sqrt{t/2}}} \right)\e.$$
Moreover, it holds that $$p(t,x,y)\geq \int_{{t}/{2}}^{{2t}/{3}}p(t-s,a^*,y)\mathbb{P}_x(\sigma_{a^*}\in ds)
\gtrsim \frac{1}{t\log{t}}\e. $$
Thus we obtain \begin{eqnarray*} p(t,x,y)&\gtrsim &\frac{1+\log{\left(\frac{1}{|y|}\sqrt{\frac{t}{2}} \right)}}{t\log{t}}\e \\
&\gtrsim&\frac{\e}{t}\left( U_t(x)U_t(y)+\frac{U_t(x)\log{|y|}}{\log{(1+t|y|)}}+\frac{U_t(y)\log{|x|}}{\log{(1+t|x|)}} \right). \end{eqnarray*}
By the symmetry, the all cases have been proved.
\end{proof}

Next, we prove Theorem \ref{d'=2,3}.
\begin{Prop}\label{on23l} Let $d\geq 3,\  d'=2$. Then, for $t>0$, we have
$$p(t,a^*,a^*)\gtrsim \frac{1}{\sqrt{t}}\wedge \frac{1}{(t+1)\left(\log{(t+1)}\right)^2}.$$
\end{Prop}
\begin{proof}
For $t>3$ and $x\in \R_{\varepsilon'}^2$ with $\sqrt{t}\leq |x|\leq 2\sqrt{t}$, by Theorem \ref{small}, Lemma \ref{hitting2} and Proposition \ref{on2weak}, we have
\begin{eqnarray*} p(t,a^*,x)&\geq &\int_{t-1}^tp(t-s,a^*,a^*)\mathbb{P}_x(\sigma_{a^*}\in ds)\\
	&\gtrsim&\int_{t-1}^t(t-s)^{-{1}/{2}}ds\ \frac{e^{-|x|_{\rho}^2/t}}{t}\frac{1+\log{|x|}}{\left( 1\!+\log{(1+t/|x|)} \right)\! \left( 1+\log{(t+|x|)} \right)}\\
	&\gtrsim& \frac{1}{t}\frac{1+\log{t}}{\left( 1+\log{(1+\sqrt{t})} \right)\left( 1+\log{(t+\sqrt{2t})} \right)}\\
	&\gtrsim& \frac{1}{(t+1)\log{(t+1)}}.\end{eqnarray*}
Thus, by the Markov property, we have
\begin{eqnarray*}
	p(2t,a^*,a^*)&\geq& \int_{\{x\in \R_{\varepsilon'}^2; \sqrt{t}\leq |x|\leq 2\sqrt{t}\}} p(t,a^*,x)^2m_p(dx)\\
	&\gtrsim&\int_{\{x\in \R_{\varepsilon'}^2; \sqrt{t}\leq |x|\leq 2\sqrt{t}\}}\frac{1}{(t+1)^2\left(\log{(t+1)}\right)^2}m_p(dx)\\
	&\asymp& \int_{\sqrt{t}}^{\sqrt{2t}}\frac{r}{(t+1)^2\left(\log{(t+1)}\right)^2}dr
	\asymp \frac{1}{(t+1)\left(\log{(t+1)}\right)^2},
\end{eqnarray*} where we used polar coordinates.
Combining this with Theorem \ref{small}, the desired estimate is proved.
\end{proof}

Next, we prove that the estimate in Proposition \ref{on23l} is sharp by using Doob's $h$-transform. First, we construct a harmonic function, which is comparable to $1$ on $\R^d_{\varepsilon}$ and $1+\log{|x|_{\rho}}$ on $\R_{\varepsilon'}^2$. In the following proposition, we use some ideas from {\cite[Theorem 2.6]{STW}}.
\begin{Prop}\label{harmonic} Let $d\geq 3$. Then, there exists a positive harmonic function $h$ on $\R_{\varepsilon}^d\cup \R_{\varepsilon'}^2\cup \{a^*\}$ such that $h \asymp 1$ on $\R_{\varepsilon}^d$ and $h(x)\asymp 1+\log{|x|_{\rho}}$ for $x\in \R_{\varepsilon'}^2$ with $|x|_{\rho}$ large enough.
\end{Prop}
\begin{proof}
	Let $R>0$ be large, $K_1:=\R_{\varepsilon}^d\cap \{|x|_{\rho}\leq R\}$ and $K_2:=\R_{\varepsilon'}^2\cap \{|x|_{\rho}\leq R\}$. By {\cite[Lemma 6.1]{GSc}}, there exists a positive harmonic function $h_1$ on $\R_{\varepsilon}^d \setminus K_1$ such that $h_1=0$ on $K_1\cup \{a^*\}$ and, for large $|x|_{\rho}$, $h_1 \asymp 1$. By {\cite[Lemma 6.1]{GSc}} again, there exists a positive harmonic function $h_2$ on $\R_{\varepsilon'}^2 \setminus K_2$ such that $h_2=0$ on $K_2\cup \{a^*\}$ and $h_2(x) \asymp \log{|x|_{\rho}}$ for large $|x|_{\rho}$.\\
	Let $K:=\{|x|_{\rho}\leq R\}$ and \begin{eqnarray*} f(x):=\left \{ \begin{split} h_1(x)\ &:&x\in \R_{\varepsilon}^d\cup \{a^*\} ,\\
	h_2(x)\ &:&x\in \R_{\varepsilon'}^{d'}\cup \{a^*\}. \end{split} \right . \end{eqnarray*}
	We take $\eta \in C^{\infty}(\R_{\varepsilon}^d\cup \R_{\varepsilon'}^2\cup \{a^*\})$ satisfying $\eta =1$ on $\{|x|_{\rho}>2R\}$ and $\eta=0$ on $K$.
	Let $$h(x):=(\eta f)(x)+\int_{\R_{\varepsilon}^d\cup \R_{\varepsilon'}^2\cup \{a^*\}}G(x,y)\Delta (\eta f)(y)dy,$$
	where $G(x,y):=\int_0^{\infty}p(t,x,y)dt$. Since $\R_{\varepsilon}^d\cup \R_{\varepsilon'}^2\cup \{a^*\}$ is non-parabolic, we have $G(x,y)<\infty$. It holds that $\Delta (\eta f)\in C_c^{\infty}(\R_{\varepsilon}^d\cup \R_{\varepsilon'}^2\cup \{a^*\})$ since $(\eta f)(x)=f(x)$ and $\Delta(\eta f)(x)=0$ for $x$ with $|x|_{\rho}>2R$. Hence, for all $x$, we have
	$\Delta h(x)=0,$ so $h$ is a harmonic function.

	For $x$ with $|x|_{\rho}>4R$, 
	\begin{eqnarray*}|f(x)-h(x)|&=&\left|\int_{\{|y|_{\rho}\leq 2R\}} G(x,y)\Delta(\eta f)(y)dy \right| \\
	&\leq& \sup_{\{|y|_{\rho}\leq2R \}}{G(x,y)} \times |\{|y|_{\rho}\leq2R \}|\times \sup{\Delta(\eta f)}\\
	&\leq& C\sup_{\{|y|_{\rho}\leq2R \}}{G(x,y)}.  \end{eqnarray*}
By using the elliptic Harnack inequality on $\R_{\varepsilon}^d\cup \{a^*\}$ and $\R_{\varepsilon'}^{d'}\cup \{a^*\}$ (see for example {\cite[Theorem 13.10]{G1}}), it holds that $|f(x)-h(x)|\leq CG(x,a^*)$ for $x$ with $|x|_{\rho}>4R$.

	Let fix $x_1\in \R_{\varepsilon}^d$ and $x_2\in \R_{\varepsilon'}^2$ with $|x_1|_{\rho}=|x_2|_{\rho}=4R$. For $x\in \R_{\varepsilon}^d$ with $|x|_{\rho}>4R$, by Lemma \ref{hitting3}, we have $\mathbb{P}_x(\sigma_{a^*}\in ds)\lesssim \mathbb{P}_{x_1}(\sigma_{a^*}\in ds)$. Furthermore, for $x\in \R_{\varepsilon'}^2$ with $|x|_{\rho}>4R$, by Lemma \ref{hitting2} and for large $R$, we have $\mathbb{P}_x(\sigma_{a^*}\in ds)\lesssim \mathbb{P}_{x_2}(\sigma_{a^*}\in ds)$. Thus, we have
	$$p(t,x,a^*)=\int_0^t p(t-s,a^*,a^*)\mathbb{P}_x(\sigma_{a^*}\in ds) \lesssim p(t,x_1,a^*)+p(t,x_2,a^*)$$ and $G(x,a^*)\lesssim G(x_1,a^*)+G(x_2,a^*)<\infty$.
	Then, $|f-h|$ is bounded on $\{ |x|_{\rho}>4R\}$ and, by the continuity of $G$, $|f-h|$ is bounded.
\end{proof}

Let $h$ be a positive harmonic function constructed as above. Define
$$H:L^2(\R_{\varepsilon}^d\cup \R_{\varepsilon'}^2\cup \{a^*\};h^2m_p)\ni f \mapsto fh \in L^2(\R_{\varepsilon}^d\cup \R_{\varepsilon'}^2\cup \{a^*\};m_p),$$
$$\mathcal{E}^h(f,f):=\mathcal{E}(fh,fh)=\int|\nabla(fh)|^2dm_p\ {\rm for}\  f\in \mathcal{F}^h:=H^{-1}\mathcal{F}.$$
Then, $H^{-1}\circ P_t \circ H$ admits a transition density $p^h(t,x,y)$  with respect to $dm_p^h:=h^2dm_p$ and $p^h(t,x,y)h(x)h(y)=p(t,x,y)$ holds ({\cite[Lemma 5.6]{GSn}}). Since $h$ is harmonic, by {\cite[Proposition 5.7]{GSn}}, we have
$$\mathcal{E}^h(f,f)=\int |\nabla f|^2dm_p^h\ \ \ {\rm for}\ f\in \mathcal{F}^h.$$
\ \\
The next lemma follows from {\cite[Lemma 4.8]{GSe}}.
\begin{Lem}\label{ext2}
	Let $q(t,x,y)$ be the transition density function with respect to $m_p$ on $\R_{\varepsilon'}^2\cup \{a^*\}$ and $q^h(t,x,y)$ be the $h$-transform of $q(t,x,y)$. Then it holds that $q^h(t,x,x)\lesssim m_p\left(\{y\in \R^2_{\varepsilon'} \cup \{a^*\}|\rho(x,y)\leq \sqrt{t} \} \right)^{-1}$ for $t>0, x\in \R^2_{\varepsilon'} \cup \{a^*\}$.
\end{Lem}

In order to get the sharp estimate of $p(t,a^*,a^*)$, we imitate the technique of the relative Faber-Krahn inequality appearing $\cite{GSs}$.
\begin{Lem}\label{FK2}
For some constant $c>0$, $\alpha_2>0$ and any ball $B:=B(x_0;R)$, let $$\Lambda _1(B,v):=\frac{c}{R^2}\left( \frac{m_p^h(B)}{v} \right)^{{2}/{d}},\ \Lambda _2(B,v):=\frac{c}{R^2}\left( \frac{m_p^h(B)}{v} \right)^{\alpha_2}.$$
Then, for any ball $B_1\subset \R_{\varepsilon}^d\cup \{a^*\}$ and a non-empty open subset $\Omega \subset B_1$, we have $$\inf_{f\in C_c^{\infty}(\Omega)\setminus \{0\}}\frac{\int_{\Omega}|\nabla f|^2dm_p^h}{\int_{\Omega}|f|^2dm_p^h}\geq \Lambda _1(B_1,m_p^h(\Omega))$$
and,  for any ball $B_2\subset \R_{\varepsilon'}^2\cup \{a^*\}$ and non-empty open subset $\Omega \subset B_2$,  we have $$\inf_{f\in C_c^{\infty}(\Omega)\setminus \{0\}}\frac{\int_{\Omega}|\nabla f|^2dm_p^h}{\int_{\Omega}|f|^2dm_p^h}\geq \Lambda _2(B_2,m_p^h(\Omega)).$$
These inequalities are called the relative Faber-Krahn inequality.
\end{Lem}
\begin{proof}
	From {\cite[Proposition 4.2]{GSs}}, for a complete weighted manifold, the relative Faber-Krahn inequality holds if the diagonal upper estimate of heat kernel holds. $h\asymp 1$ and $q(t,x,x)\lesssim t^{-{d}/{2}}$ on $\R_{\varepsilon}^d\cup \{a^*\}$, where $q$ is a heat kernel with respect to $m_p^h$ on $\R_{\varepsilon}^d\cup \{a^*\}$, so the first inequality holds. The second inequality follows from Lemma \ref{ext2}.\end{proof}

\begin{Thm}\label{RFKineq}
 Let $\alpha:=\alpha_2\wedge \frac{2}{d}$. For $B:=B(x_0; R)\subset \R_{\varepsilon}^d\cup \R_{\varepsilon'}^2\cup \{a^*\}$ and any open subset $\Omega \subset B$, it holds that
$$\inf_{f\in C_c^{\infty}(\Omega)\setminus \{0\}}\frac{\int_{\Omega}|\nabla f|^2dm_p^h}{\int_{\Omega}|f|^2dm_p^h}\geq \Lambda(B,m_p^h(\Omega)),\ {\rm where}\  \Lambda(B,v):=\frac{c}{R^2}\left(\frac{F(B)}{v} \right)^{\alpha},\ $$
\begin{equation*}F(B):=\left \{ \begin{split} &m_p^h(B)&:&\ B\subset \R_{\varepsilon_i}^i {\rm for}\  i=2,d,\\  &m_p^h(\{x\in\R_{\varepsilon'}^2\cup\{a^*\}\ |\ \rho(x,y_2)\leq R\}) &:&\ a^*\in B\ {\rm and\ large}\ R. \end{split} \right.  \end{equation*}
Here, $\varepsilon_2:=\varepsilon', \varepsilon_d:=\varepsilon$ and $y_2\in \R_{\varepsilon'}^2\cup \{a^*\}$ with $|x_0|_{\rho}=|y_2|_{\rho}$.
\end{Thm}

\begin{proof}
When $B\subset \R_{\varepsilon_i}^i {\rm for}\  i=2,d$, the estimate holds by Lemma \ref{FK2}. When $a^*\in B:=B(x_0; R)\subset \R_{\varepsilon}^d\cup \R_{\varepsilon'}^2\cup \{a^*\}$ for large $R>0$, for any open subset $\Omega \subset B$ and $f\in C_c^{\infty}(\Omega)\setminus \{0\}$, it holds that $f|_{\R_{\varepsilon}^d\cup \{a^*\}}\in C_c^{\infty}\left(\Omega \cap (\R_{\varepsilon}^d\cup \{a^*\})\right)\setminus \{0\}$ and $f|_{\R_{\varepsilon'}^2\cup \{a^*\}}\in C_c^{\infty}\left(\Omega \cap (\R_{\varepsilon'}^2\cup \{a^*\})\right)\setminus \{0\}$.
For $i=2,d$, fix $y_i\in \R_{\varepsilon_i}^i\cup \{a^*\}$ satisfying $|y_i|_{\rho}=|x_0|_{\rho}$ and $B^i:=\{x\in \R_{\varepsilon_i}^i\cup \{a^*\}; \rho(x,y_i) \leq 3R+2\varepsilon +2\varepsilon' \}$.\\
$($1$)$ For $x_0\in \R_{\varepsilon'}^2$, we have (see Figure \ref{lemballfig1})
\begin{eqnarray*}B\cap (\R_{\varepsilon'}^2\cup \{a^*\})\subset B(y_2;|x_0|_{\rho}+2\varepsilon'+|x_0|_{\rho}+R)\cap (\R_{\varepsilon'}^2\cup \{a^*\})\subset B^2,\\
B\cap (\R_{\varepsilon}^d\cup \{a^*\})\subset B(y_d;2\varepsilon+2R)\cap (\R_{\varepsilon}^d\cup \{a^*\})\subset B^d\ {\rm if}\ R-|x_0|_{\rho}\leq |y_d|_{\rho},\\
B\cap (\R_{\varepsilon}^d\cup \{a^*\})\subset B(y_d;2R)\cap (\R_{\varepsilon}^d\cup \{a^*\})\subset B^d\ {\rm if}\  R-|x_0|_{\rho}> |y_d|_{\rho}. \hspace{8mm} \end{eqnarray*}

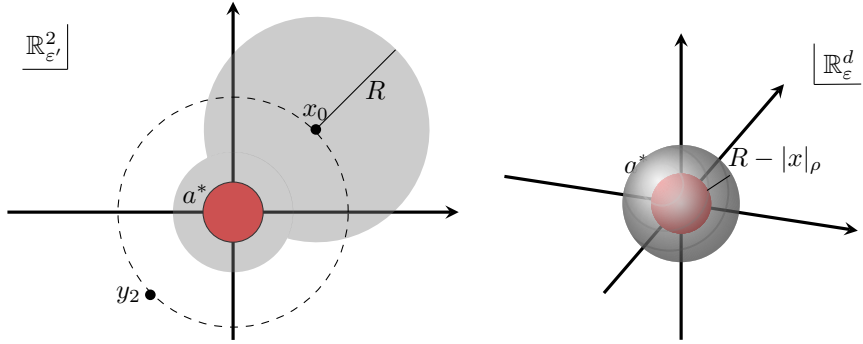
\begin{figure}[h]
 \begin{tikzpicture}
 \draw[->,>=stealth,very thick] (-3,0)--(3,0)node[above]{};
 \draw[->,>=stealth,very thick] (0,-1.7)--(0,2.8)node[right]{};
 \draw (-2.2,2.5)--(-2.2,1.9);
 \draw (-2.2,1.9)--(-2.8,1.9);
 \draw (-2.5,2.2)node{$\R^{2}_{\varepsilon'}$};
 \draw[fill=red!80](0,0)circle(0.4);
 \fill[gray, opacity =0.4] (0.723,-0.35) arc (255:554:1.5)--(-0.348,0.719) (-0.348,0.719) arc(-244.55:-383.5:0.8)--(0.723,-0.35);
 \fill[gray, opacity = 0.4] (0,0)circle(0.8);
\fill (1.1,1.1) circle (2pt) coordinate (x) circle node [above] {$x_0$} ;
\draw[dashed] (0,0)circle(1.53);
\fill (-1.1,-1.1) circle (2pt) coordinate (y2) circle node [left] {$y_2$} ;
 \draw (-0.2,0)node[above left]{$a^*$} ;
\draw[] (1.1,1.1) -- node[right]{$R$} (2.16,2.16);
\end{tikzpicture}
\tdplotsetmaincoords{65}{20}
\begin{tikzpicture}[tdplot_main_coords]
 \draw[->,>=stealth,very thick] (-2.5,0,0)--(2.5,0,0)node[above]{};
 \draw[->,>=stealth,very thick] (0,-3,0)--(0,4,0)node[right]{};
 \draw[->,>=stealth,very thick] (0,0,-2)--(0,0,2.5)node[right]{};
 \draw (1.7,0.5,1.8)--(1.7,0.5,2.4);
 \draw (1.7,0.5,1.8)--(2.1,1.23,1.55);
 \draw (1.8,1.2,1.8)node{$\R^{d}_{\varepsilon}$};
\node[circle,shading=ball, outer color=red, inner color=pink, minimum width=8mm][label=above left:$a^*$] (ball) at (0,0) {};
\shade[ball color = gray!60, opacity = 0.4] (0,0,0) circle (0.78cm);
\draw[] (0.33,0.1,0.2) -- node[above right]{$R-|x|_{\rho}$} (0.6,0.25,0.4);
\end{tikzpicture} 
\caption{locations of $y_2$ and $B=B(x_0;R)$}
\label{lemballfig1}
\end{figure}

$($2$)$ For $x_0\in \R_{\varepsilon}^d$, by the same way as $($1$)$, it holds that $B\cap (\R{\varepsilon}^d\cup \{a^*\})\subset B^d$ and $B\cap (\R{\varepsilon'}^2\cup \{a^*\})\subset B^2$.\\
Hence, for $i=2$ or $d$ satisfying $\int_{\Omega \cap \R_{\varepsilon_i}^i}|f|^2dm_p^h\geq \frac{1}{2} \int_{\Omega}|f|^2dm_p^h$, by Lemma \ref{FK2} and $m_p^h(B^i)\geq m_p^h(\Omega \cap \R_{\varepsilon_i}^i)$, we have
\begin{eqnarray} \nonumber \int_{\Omega} |\nabla f|^2dm_p^h &\geq& \int_{\Omega \cap \R_{\varepsilon_i}^i}|\nabla f|^2dm_p^h \\
\nonumber &\geq& \Lambda_i\left( B^i, m_p^h(\Omega \cap \R_{\varepsilon_i}^i)  \right)\int_{\Omega \cap \R_{\varepsilon_i}^i}|f|^2dm_p^h\\
\nonumber &\geq& \frac{c}{(3R+2\varepsilon +2\varepsilon')^2}\left( \frac{m_p^h(B^i)}{m_p^h(\Omega \cap \R_{\varepsilon_i}^i)} \right)^{\alpha} \frac{1}{2} \int_{\Omega}|f|^2dm_p^h\\
&\gtrsim& \frac{c}{R^2}\left( \frac{m_p^h(B^i)}{m_p^h(\Omega)} \right)^{\alpha}\int_{\Omega}|f|^2dm_p^h. \label{eq:FKa1}
\end{eqnarray}
Hence, the proof is finished if $i=2$. If $i=d$, we have
\begin{eqnarray}
\nonumber m_p^h(B^d)&\geq& m_p(B(y_d;3R)\cap \R_{\varepsilon}^d)\geq m_p(B(y'_d;R)\cap \R_{\varepsilon}^d)\asymp  R^d\\
\nonumber & \geq& (1+\log{(5R)})^2(8R)^2 \\
\nonumber &\geq& \int_{B(y_2;4R)\cap \R_{\varepsilon'}^2}(1+\log{(|y_2|_{\rho}+3R+2\varepsilon+2\varepsilon')})^2\\
&\geq& m_p^h(B^2)\geq m_p^h((B(y_2;R)\cap \R_{\varepsilon'}^2), \label{eq:FKa2}
\end{eqnarray}
where $y'_d$ is the point with $|y_d'|_{\rho}=2R$ on the line joining $a^*$ and $y_d$. 
Therefore, by (\ref{eq:FKa1}) and (\ref{eq:FKa2}), the desired inequality holds.
\end{proof}

\begin{Prop}\label{on23} Let $d\geq 3$ and $d'=2$. Then, it holds that for $t>0$, 
$$p(t,a^*,a^*)\asymp \frac{1}{\sqrt{t}}\wedge \frac{1}{(t+1)\left(\log{(t+1)}\right)^2}.$$
\end{Prop}

\begin{proof}
The lower estimate is given in Proposition \ref{on23l}, so we prove the upper estimate.

By the same proof of {\cite[Theorem 5.2]{G2}}, for large $T>0$, $t\in [T,\infty)$ and $x,y\in \R_{\varepsilon}^d\cup \R_{\varepsilon'}^2\cup \{a^*\}$, it holds that \begin{equation}p^h(t,x,y)\lesssim {(t\wedge R^2)^{-1/\alpha}\left( \frac{F(B(x;R))^{\alpha}}{R^2}\frac{F(B(y;R))^{\alpha}}{R^2} \right)^{-1/2\alpha}}. \label{eq:RFKU}\end{equation}
Indeed, in \cite{G2}, Grigor'yan proved (\ref{eq:RFKU}) for $t>0$ on a smooth connected non-compact complete Riemannian manifold. In the proof, it is used that $|\nabla \hat{\rho}|\leq 1$, where $\hat{\rho}$ is a Riemannian distance. In our setting, we consider the space attached by two manifolds on which $|\nabla \rho|\leq 1$ still holds. Hence (\ref{eq:RFKU}) holds by the proof of {\cite[Theorem 5.2]{G2}}.

We take $R:=\sqrt{t}$ and large $t$, by Theorem \ref{RFKineq}, we have
\begin{eqnarray}p^h(t,a^*,a^*\nonumber)&\lesssim & {t^{-1/\alpha}\left( \frac{m_p^h(\{x\in\R_{\varepsilon'}^2\cup\{a^*\};|x|_{\rho}\leq \sqrt{t} \})^{\alpha}}{t}\right)^{-1/\alpha}}\\
&=&{m_p^h(\{x\in\R_{\varepsilon'}^2\cup\{a^*\};|x|_{\rho}\leq \sqrt{t} \})^{-1}}.\label{eq:on23c1}\end{eqnarray}
Let $\tilde{B}:=B((3\sqrt{t}/2,0);\sqrt{t}/2)\cap \R_{\varepsilon'}^2$, then we obtain
\begin{eqnarray}
m_p^h(\{x\in\R_{\varepsilon'}^2\cup\{a^*\};|x|_{\rho}\leq \sqrt{t} \}) &\geq& m_p^h(\tilde{B}) \label{eq:on23c2}\\
\nonumber &\asymp& \int_{\{|x-(3\sqrt{t}/2,0)\leq\sqrt{t}/2\}|}(1+\log{|x|_{\rho}})^2dx \\
\nonumber  &\geq &(1+\log{\sqrt{t}/2})^2|\tilde{B}|\gtrsim (t+1)(\log{(t+1)})^2 .
\end{eqnarray}
By (\ref{eq:on23c1}), (\ref{eq:on23c2}), $p^h(t,a^*,a^*)=p(t,a^*,a^*)h(a^*)^2$ and Theorem \ref{small}, the upper estimate holds.
\end{proof}

\begin{proof}[Proof of Theorem \ref{d'=2,3}]
By comparing with the heat kernel on $\R^d\# (\R^2\times S_{\varepsilon}^{d-2})$ (\cite{GSc}), the upper estimates can be proved in the same way as the proof of Proposition \ref{off3u}.
We prove the lower estimates. Let $T>0$ be large and $t\in [T,\infty)$.\\
{\bf Step1} (the estimate of $p(t,x,a^*)$)\\
$($1$)$ For $x\in \R_{\varepsilon}^d\cup \{a^*\}$, by Proposition \ref{on23}, Theorem \ref{small}, Lemma \ref{GSh3}, Lemma \ref{hitting3}, and Lemma \ref{elem3}, we have
\begin{eqnarray*}p(t,x,a^*)&\geq &\int_{0}^{{t}/{2}} + \int_{t-2}^{t}p(t-s,a^*,a^*)\mathbb{P}_x(\sigma_{a^*}\in ds) \\
&\asymp&\frac{1}{t(\log{t})^2}\frac{1}{|x|^{d-2}}\ex+\frac{1}{t^{{d}/{2}}}\ex .\end{eqnarray*}
$($2$)$ For $y\in \R_{\varepsilon'}^2\cup \{a^*\}$ with $|y|<\sqrt{t}/2$, by Proposition \ref{on23}, Theorem \ref{small}, Lemma \ref{GSh2} and Lemma \ref{hitting2}, we have
\begin{eqnarray*} p(t,y,a^*)&\geq &\int_{0}^{{t}/{2}} + \int_{t-2}^{t}p(t-s,a^*,a^*)\mathbb{P}_y(\sigma_{a^*}\in ds)\\
&\asymp& \frac{1}{t(\log{t})^2}\left(1-\frac{\log{|y|}}{\log{\sqrt{t/2}}}\right)+\frac{|y|}{t(\log{t})^2}\asymp \frac{1}{t(\log{t})^2}\ey.  \end{eqnarray*}
$($3$)$ For $y\in \R_{\varepsilon'}^2\cup \{a^*\}$ with $|y|\geq \sqrt{t}/2$, by Theorem \ref{small} and Lemma \ref{GSh2}, we have
\begin{eqnarray*} p(t,y,a^*)&\geq &\int_{t-2}^{t}p(t-s,a^*,a^*)\mathbb{P}_y(\sigma_{a^*}\in ds)\\
&\asymp& \frac{1+\log{|y|}}{(1+\log{(1+t/|y|)})(1+\log{(t+|y|)})}\frac{(|y|+t)^{{1}/{2}}}{t^{{3}/{2}}}\ey\\
&\gtrsim& \frac{1}{t(\log{t})}\ey \gtrsim \frac{1}{t(\log{t})^2}\ey. \end{eqnarray*}
Since ${H}_t(y)\leq (\log{(1+\varepsilon')}^{-2})+(2\log{(1+\varepsilon')})^{-1}\asymp 1$, these estimates are sharp.
\\
{\bf Step2} (the proof of Theorem \ref{d'=2,3} $($i$)$)\\
$($1$)$ For $x,y\in \R_{\varepsilon}^d$ with $1\leq |x|_{\rho} \wedge |y|_{\rho}$, by (\ref{eq:part kernel}), Step1 and Lemma \ref{GSh3}, we have
\begin{eqnarray*} p(t,x,y)\geq \int_0^{{t}/{2}}p(t-s,a^*,x)\mathbb{P}_{y}(\sigma_{a^*}\in ds)+ p_{\R_{\varepsilon}^d}(t,x,y)\hspace{28mm}\\
\gtrsim \frac{1}{t(\log{t})^2|x|^{d-2}}\ex \mathbb{P}_y(\sigma_{a^*} \leq t/2)+(|x|_{\rho}\wedge 1)(|y|_{\rho}\wedge 1)\frac{1}{t^{{d}/{2}}}\e \hspace{-6mm}\\
\asymp \frac{1}{t(\log{t})^2|x|^{d-2}|y|^{d-2}}\exy +\frac{1}{t^{{d}/{2}}}\e .\hspace{23mm}
\end{eqnarray*}
$($2$)$ For $x,y\in \R_{\varepsilon}^d$ with $|x|_{\rho} <1$, by Step1, Lemma \ref{hitting3}, (\ref{eq:part kernel}), Lemma \ref{GSh3}, Lemma \ref{elem3} and Lemma \ref{rholem} (iii), we have
\begin{eqnarray*} p(t,x,y)\geq \int_0^{{t}/{2}}+ \int_{t-1}^t p(t-s,a^*,x)\mathbb{P}_{y}(\sigma_{a^*}\in ds)+p_{\R_{\varepsilon}^d}(t,x,y) \hspace{20mm} \\
\asymp \frac{\mathbb{P}_y(\sigma_{a^*} \leq t/2)}{t(\log{t})^2|x|^{d-2}}+\frac{\ey}{t^{d/2}+t^{3/2}|y|^{{(d-3)}/2}}+\frac{(|x|_{\rho}\wedge 1)(|y|_{\rho}\wedge 1)}{t^{{d}/{2}}}\e \hspace{2mm}\\
\gtrsim \frac{\exy}{t(\log{t})^2|x|^{d-2}|y|^{d-2}}+\frac{\exy}{t^{d/2}} +\frac{(|x|_{\rho}\wedge 1)(|y|_{\rho}\wedge 1)}{t^{d/2}}\e\hspace{1mm}\\
\gtrsim \frac{1}{t(\log{t})^2|x|^{d-2}}\exy +\frac{1}{t^{{d}/{2}}}\e. \hspace{36mm}
\end{eqnarray*}
{\bf Step3} (the proof of Theorem \ref{d'=2,3} $($ii$)$)\\
$($1$)$ For $x,y\in \R_{\varepsilon'}^2$ with $|x|_{\rho}\leq 1,  |y|_{\rho}\leq \sqrt{t}/2$, by Theorem \ref{small}, Lemma \ref{GSh2}, (\ref{eq:part kernel}), and (\ref{eq:on3llll}), we have
\begin{eqnarray*} p(t,x,y)&\geq& \int _{t-2}^{t-1} p(t-s,x,a^*)\mathbb{P}_y(\sigma_{a^*}\in ds)+ p_{\R_{\varepsilon'}^2}(t,x,y)\\
&\gtrsim& \int _{t-2}^{t-1} \frac{e^{-|x|_{\rho}^2/(t-s)}ds}{\sqrt{t-s}} \frac{\log{(1+|y|)}}{t(\log{t})^2}+\frac{(|x|_{\rho}\wedge \!1)(|y|_{\rho}\wedge 1)}{t}\e\\
&\asymp& \frac{\log{(1+|y|)}}{t(\log{t})^2}+\frac{|x|_{\rho}(|y|_{\rho}\wedge 1)}{t}\e\\
&\asymp& \frac{\log{(1+|x|)}\log{(1+|y|)}}{t(\log{t})^2}+\frac{|x|_{\rho}(|y|_{\rho}\wedge 1)}{t}\e\\
&\gtrsim& \frac{\log{(1+|x|)}\log{(1+|y|)}}{t(\log{t})^2}\e\\
&\asymp& \frac{\log{(1+|x|)}\log{(1+|y|)}}{t(\log{(1+t|x|)})(\log{(1+t|y|)})}\e.
\end{eqnarray*}
$($2$)$ For $x,y\in \R_{\varepsilon'}^2$ with $1\leq |x|_{\rho}\leq \sqrt{t}/2, |y|_{\rho}\leq \sqrt{t}/2$, by (\ref{eq:part kernel}), we have
\begin{eqnarray*}p(t,x,y) \geq p_{\R_{\varepsilon'}^2}(t,x,y) \gtrsim \frac{1}{t}\e \gtrsim \frac{\log{(1+|x|)}\log{(1+|y|)}}{t(\log{t})^2}\e\\
\asymp \frac{\log{(1+|x|)}\log{(1+|y|)}}{t(\log{(1+t|x|)})(\log{(1+t|y|)})}\e.
\end{eqnarray*}
$($3$)$ For $x,y\in \R_{\varepsilon'}^2$ with $|x|_{\rho}\leq  1,\  \sqrt{t}/2 \leq |y|_{\rho}$, by Theorem \ref{small}, Lemma \ref{hitting2} and Lemma \ref{rholem} (iii), we have
\begin{eqnarray*}p(t,x,y) &\geq&\int_{t-2}^{t-1}p(t-s,a^*,x)\mathbb{P}_y(\sigma_{a^*}\in ds)\\
&\gtrsim& \frac{1}{t\log{t}}\exy \gtrsim  \frac{\log{(1+|x|)}}{t\log{t}}\frac{|y|_{\rho}^2}{t}\exy\\
&\gtrsim& \frac{\log{(1+|x|)}\log{(1+|y|)}}{t(\log{t})^2}\exy \\
&\asymp& \frac{\log{(1+|x|)}\log{(1+|y|)}}{t(\log{(1+t|x|)})(\log{(1+t|y|)})}\e.
\end{eqnarray*}
$($4$)$ For $x,y\in \R_{\varepsilon'}^2$ with $1\leq |x|_{\rho}\leq \sqrt{t}/2, \sqrt{t}/2 <|y|_{\rho}$, by (\ref{eq:part kernel}), we have
\begin{eqnarray*}p(t,x,y) &\geq& p_{\R_{\varepsilon'}^2}(t,x,y) \gtrsim \frac{1}{t}\e \gtrsim \frac{\log{(1+|x|)}}{t\log{t}}\e\\
&\asymp& \frac{\log{(1+|x|)}\log{(1+|y|)}}{t(\log{(1+t|x|)})(\log{(1+t|y|)})}\e.
\end{eqnarray*}
$($5$)$ For $x,y\in \R_{\varepsilon'}^2$ with $\sqrt{t}/2 \leq |x|_{\rho} \wedge |y|_{\rho}$, by (\ref{eq:part kernel}), we have
\begin{eqnarray*}p(t,x,y) \geq p_{\R_{\varepsilon'}^2}(t,x,y) \gtrsim \frac{1}{t}\e \gtrsim \frac{\log{(1+|x|)}\log{(1+|y|)}}{t(\log{(1+t|x|)})(\log{(1+t|y|)})}\e.
\end{eqnarray*}
{\bf Step4} (the proof of Theorem \ref{d'=2,3} $($iii$)$)\\
$($1$)$ For $x\in \R_{\varepsilon}^d, y\in \R_{\varepsilon'}^2$ with $|x|_{\rho}<1$,  by Theorem \ref{small} and Lemma \ref{hitting2}, we have
\begin{eqnarray*}p(t,x,y)&\geq &\int_{t-2}^{t-1}p(t-s,a^*,x)\mathbb{P}_y(\sigma_{a^*}\in ds)\\
&\gtrsim& \frac{1}{t\log{t}}\frac{1+\log{|y|}}{1+\log{(t+|y|)}}\e\\
&\gtrsim& \frac{1}{t(\log{t})^2}\e\\
&\gtrsim& \left(\frac{1}{t(\log{t})^2|x|^{d-2}}+\frac{{H}_t(y)}{t^{d/2}}\right)\e.
\end{eqnarray*}
$($2$)$For $x\in \R_{\varepsilon}^d, y\in \R_{\varepsilon'}^2$ with $1\leq |x|_{\rho}<|x|\leq \sqrt{t}/2,|y|_{\rho}\leq 1$,  by Step1, Theorem \ref{small}, Lemma \ref{GSh3}, Lemma \ref{hitting3} and ${H}_t(y)\lesssim 1$, we have
\begin{eqnarray*}p(t,x,y)&\geq &\int_0^{{t}/{2}}+\int_{t-2}^{t-1}p(t-s,a^*,y)\mathbb{P}_x(\sigma_{a^*}\in ds)\\
&\gtrsim& \frac{1}{t(\log{t})^2|x|^{d-2}}\e+ \int_{t-2}^{t-1}\frac{e^{-|y|_{\rho}^2/(t-s)}}{\sqrt{t-s}}ds\frac{\ex}{t^{d/2}}\\
&\gtrsim& \left(\frac{1}{t(\log{t})^2|x|^{d-2}}+\frac{{H}_t(y)}{t^{d/2}}\right)\e.
\end{eqnarray*}
$($3$)$ For $x\in \R_{\varepsilon}^d, y\in \R_{\varepsilon'}^2$ with $1\leq |x|_{\rho}<|x|\leq \sqrt{t}/2,\ 1\leq |y|_{\rho}<|y|\leq \sqrt{t}/2$, by Step1 and Lemma \ref{GSh3}, we have
\begin{eqnarray*}p(t,x,y)\geq \int_0^{{t}/{2}}p(t-s,a^*,y)\mathbb{P}_x(\sigma_{a^*}\in ds)
\gtrsim \frac{1}{t(\log{t})^2|x|^{d-2}}\e. \end{eqnarray*}
Furthermore, by Step1 and Lemma \ref{GSh2}, we have
\begin{eqnarray*}p(t,x,y)\geq \int_0^{{t}/{2}}p(t-s,a^*,x)\mathbb{P}_y(\sigma_{a^*}\in ds)
\gtrsim \frac{\e}{t^{d/2}}\left(1-\frac{\log{|y|}}{\log{\sqrt{t/2}}} \right)\\
\gtrsim \frac{{H}_t(y)}{t^{d/2}}\e. \hspace{67.5mm}
\end{eqnarray*}
$($4$)$ For $x\in \R_{\varepsilon}^d, y\in \R_{\varepsilon'}^2$ with $1\leq |x|_{\rho}<|x|\leq \sqrt{t}/2,\  \sqrt{t}/2\leq |y|$,  by Step1 and Lemma \ref{GSh3}, we have
\begin{eqnarray*}p(t,x,y)\geq \int_0^{{t}/{2}}p(t-s,a^*,y)\mathbb{P}_x(\sigma_{a^*}\in ds)
\gtrsim \frac{1}{t(\log{t})^2|x|^{d-2}}\e. \end{eqnarray*}
Furthermore, by Step1, Lemma \ref{GSh2} and ${H}_t(y)=(\log{(1+|y|)})^{-2}$,
\begin{eqnarray*}p(t,x,y)\geq \int_0^{{t}/{2}}p(t-s,a^*,x)\mathbb{P}_y(\sigma_{a^*}\in ds)
\gtrsim \frac{\e}{t^{d/2}\log{|y|}}
\gtrsim \frac{{H}_t(y)}{t^{d/2}}\e.
\end{eqnarray*}
$($5$)$ For $x\in \R_{\varepsilon}^d, y\in \R_{\varepsilon'}^2$ with $\sqrt{t}/2 \leq|x|$, by Step1, Lemma \ref{GSh3}, we have
\begin{eqnarray*}p(t,x,y)\geq \int_0^{{t}/{2}}p(t-s,a^*,y)\mathbb{P}_x(\sigma_{a^*}\in ds)\asymp \frac{1}{t(\log{t})^2|x|^{d-2}}\e,  \end{eqnarray*}
\begin{eqnarray}p(t,x,y)\geq \int_0^{{t}/{2}}p(t-s,a^*,x)\mathbb{P}_y(\sigma_{a^*}\in ds)\asymp \frac{\e}{t^{d/2}}\mathbb{P}_y(\sigma_{a^*}\leq t/2).\label{eq:off23s4}  \end{eqnarray}
By the same way as in Step 4 $($3$)$,$($4$)$, the right hand side of $(\ref{eq:off23s4}) $ is larger than ${H}_t(y)\e/t^{d/2}$ up to a constant multiple. \\
By the symmetry, we have proved all the cases and complete the proof of Theorem \ref{d'=2,3}.\end{proof}

\begin{Rem}
	\rm{We already proved Theorem \ref{d'=1} for the case of $d'=1,d\geq 3$ in Section \ref{lsec1}. Since it is the mixed case of transient and recurrent, it can also be proved by the same way as in this section.}
\end{Rem}

\end{document}